\documentclass[3p]{elsarticle}
\usepackage[english]{babel}
\usepackage{amsmath}
\usepackage{amsfonts}
\usepackage{amssymb}
\usepackage{amsthm}
\usepackage{graphicx}
\usepackage{array}
\usepackage{tabularx}
\usepackage{color}
\usepackage{tcolorbox}

\newtheorem{theorem}{Theorem}

\newtheorem{lemma}{Lemma}

\begin{document}

\begin{frontmatter}

\title{Bounds for Wave Speeds in the Riemann Problem: \\
Direct Theoretical Estimates}


\author[UNITN]{E. F. Toro}
\ead{eleuterio.toro@unitn.it}
\address[UNITN]{Laboratory of Applied Mathematics, DICAM, University of Trento, Italy}

\author[MATH]{L. O. M{\"u}ller \corref{cor1}}
\ead{lucas.muller@unitn.it}
\address[MATH]{Department of Mathematics, University of Trento, Italy}

\author[UNITN]{A. Siviglia}
\ead{annunziato.siviglia@unitn.it}

\cortext[cor1]{Corresponding author}

\begin{abstract}

In this paper we provide bound estimates for the two fastest wave speeds emerging from the solution of the Riemann problem 
for three well-known hyperbolic systems, namely the Euler equations of gas dynamics, the shallow water equations and the blood flow equations for arteries. Several approaches are presented, all being direct, that is non-iterative. The resulting bounds range from crude but simple estimates to accurate but sophisticated estimates that make limited use of information from the solution of the Riemann problem. Through a carefully chosen suite of test problems we asses our wave speed estimates against exact solutions and against previously proposed wave speed estimates.  The results confirm that the derived theoretical bounds are actually so, from below and above, for minimal and maximal wave speeds respectively. The results also show that popular previously proposed estimates do not bound the true wave speeds in general. 
Applications in mind, but not pursued here,  include (i) reliable implementation of the Courant condition to determine a stable time step in all explicit methods for hyperbolic equations; (ii) use in local time stepping algorithms and (iii) construction of  HLL-type numerical fluxes for hyperbolic equations. 

\end{abstract}

\end{frontmatter}

\section{Introduction}

Several computational approaches for approximating systems of hyperbolic equations such as
\begin{equation}                            \label{cl}
      \partial_{t} {\bf Q}(x,t) + \partial_{x} {\bf F} ({\bf Q}(x,t)) = {\bf 0}  \;
\end{equation}
in one, or multiple space dimensions, require the estimation of wave speeds present in the system at every time step. This is certainly the case for all explicit methods of the form 
\begin{equation}                               \label{fv}
      {\bf Q}^{n+1}_{i}={\bf Q}^{n}_{i} - \frac{\Delta t}{\Delta x}[{\bf F}_{i+\frac{1}{2}}-{\bf F}_{i-\frac{1}{2}}] \;,
\end{equation}
in which cell averages ${\bf Q}^{n}_{i}$ are updated to ${\bf Q}^{n+1}_{i}$ by means of intercell  numerical fluxes ${\bf F}_{i+\frac{1}{2}}$. The time-marching scheme
(\ref{fv}) needs a reliable estimation of the time step $\Delta t$, under the Courant condition
\begin{equation}                              \label{deltaT}
      \Delta t = C_{cfl} \frac{\Delta x}{S_{max}^{n}}  \;,
\end{equation}
which in turn requires the mesh size $\Delta x$ and the maximum wave speed $S_{max}^{n}$ present at time level $n$ throughout the computational domain, including cells involved in boundary conditions. In addition, there will be a CFL coefficient $C_{cfl}$ that depends on the linearised stability analysis for each particular method in use. This is necessary at every time $n$ in order to advance the solution in time to the next time level $n+1$.  In practical applications one chooses $S_{max}^{n}$ by simply evaluating the eigenvalues of the system. For example, for the one-dimensional Euler equations with eigenvalues $\lambda_{1}=u-c$, $\lambda_{3}=u$, $\lambda_{3}=u+c$, where $u$ is particle velocity and $c$ is sound speed, one uses
\begin{equation}                              \label{Smax}
     S_{max}^{n} = max_{i} \left\{|u^{n}_{i}|+c^{n}_{i}) \right\}  \;.
\end{equation}
Needless to say  estimate (\ref{Smax}) will not bound the wave speeds emerging from the interaction of the data at each cell interface.  This is one of the motivations for the present work. The above considerations regarding wave speeds and size of the time step are relevant to various numerical approaches, which include amongst others, the large class of finite volume methods \cite{Godlewski:1996a}, \cite{Toro:2009a}, \cite{LeVeque:2002a} and discontinuous Galerkin finite element methods \cite{Cockburn:1991a}, \cite{Cockburn:1998a}, \cite{Dumbser:2006a}, \cite{Dumbser:2007c}, \cite{Luo:2007a}, \cite{Luo:2008a}, \cite{Guermond:2016a}. The choice of the numerical flux ${\bf F}_{i+\frac{1}{2}}$ ranges from centred (non-upwind) methods to Godunov-type methods (upwind), with incomplete and complete \cite{Toro:2016b} Riemann solvers. All of them require a reliable choice of time step in (\ref{fv}) .

As already stated, the CFL coefficient $C_{cfl}$ in (\ref{deltaT}) depends on the linearised stability analysis for each particular method in use. For most well known methods, such as Godunov's method \cite{Godunov:1959a} and the Lax-Wedroff method \cite{Lax:1960a},  $C_{cfl} =1$, but there are other methods for which $C_{cfl} <1$. For example, for the Godunov centred method \cite{Godunov:1962a} $C_{cfl} =\frac{1}{2}\sqrt{2}$. The CFL coefficient represents the fraction of the cell width $\Delta x$ that a wave of speed $S_{max}^{n}$  can traverse in time  $\Delta t$ to preserve stability of the numerical method. Underestimation of the wave speed will result in violation of the CFL condition, and in principle, in instability of the numerical method. Surprisingly, in practice there tends to be a lot of uncertainties as to the correct value of the speed $S_{max}^{n}$. A usual devise to deal with such uncertanties is to use CFL coefficients that are clearly smaller than the upper limit value of $C_{cfl}$. For example, a common practice is to take $C_{cfl}$ to be $80\%$ or $90\%$ of the upper limit.  It is usual to use very unreliable estimates for the wave speeds and if the code crashes one re-starts the computations with a smaller value for the CFL coefficient. This is not a very satisfactory state of affairs. For low-order methods, dominated by large truncation errors that result in large numerical diffusion, the effect of underestimating the speeds might not be obvious. Higher-order methods will be more sensitive to underestimated wave speeds, or overestimated time steps. The main issue is then to obtain reliable estimates, ideally upper bounds, for the maximum wave speed. This means one must seek reliable estimates of all local wave speeds $S_{i+\frac{1}{2}}^{n}$ resulting from the interaction of data either side 
of the interface $i+\frac{1}{2}$. 

There are other computational aspects that also require a reliable estimate of local waves speeds, as pointed out by Guermond and Popov \cite{Guermond:2016a}.  Wave speed estimates enter in artificial viscosity methods; one requires a reliable $S_{max}^{n}$ to ensure that all entropy inequalities are satisfied. Reliable choices of local wave speeds $S_{i+\frac{1}{2}}^{n}$ are required in local time stepping schemes \cite{Dumbser:2007c}, \cite{Mueller:2016b}.  There are several numerical methods that in addition to requiring a reliable choice for $\Delta t$, require reliable estimates for all maximal and minimal wave speeds. The Lax-Friedrichs method, a centred method, requires a grid speed  $S^{n}=\frac{\Delta x}{\Delta t}$. A more sophisticated method is the Rusanov method \cite{Rusanov:1961a}, sometimes called the local Lax-Friedrichs method; this is an upwind method and requires local  wave speeds $S_{i+\frac{1}{2}}^{n}<\frac{\Delta x}{\Delta t}$ at each cell interface. Increasing the level of sophistication, the HLL method \cite{Harten:1983b} requires two local wave speeds per cell interface, which for $2 \times 2$ systems are enough. For larger systems one requires more local wave speeds, such as in the HLLC method \cite{Toro:1994c}. See also \cite{Dumbser:2016a}. 

Despite the importance of choosing reliable bounds for the minimal and maximal waves speeds, locally and globally, there is a lack of attention to this issue. Some works on this topic have been motivated by the design of HLL-type fluxes, see \cite{Davis:1988a}, \cite{Einfeldt:1988a}, \cite{Toro:1994a}, \cite{Batten:1997a}. More recent works motivated by artificial viscosity in continuous finite element methods has been reported by Guermond and Popov
\cite{Guermond:2016a}, who appear to be the first to have addressed this subject, in the context of the Euler equations. They essentially  proposed two ways of estimating bounds for the local maximal wave speed $S_{i+\frac{1}{2}}^{n}$, a direct, non-iterative method and an iterative but more accurate method. 

In the present paper we first review all the main existing approaches to estimate wave speeds. Then we propose new theoretical ways to estimate bounds for both the maximal and minimal  wave speeds at each intercell boundary; all proposed methods are direct (non-iterative). We do so for three hyperbolic systems of practical interest, namely the Euler equations of gas dynamics, the shallow water equations and the blood flow equations for arteries. Then by choosing appropriate test problems we compare our estimates against existing ones and against exact values obtained from the exact solution of the corresponding Riemann problem. It is shown that most existing methods fail, while all our estimates are confirmed to constitute bounds, as expected. Then amongst the proposed methods, some are more accurate than others and some are simpler than others. Inaccuracy of bounds will have a bearing on efficiency.

The rest of this paper is structured as follows. In Sect 2, via the Euler equations, we set the scene for the paper and provide exact wave speed reference solutions. 
In Sect. 3 we define the central aim of this paper and review existing estimates for the minimal and maximal wave speeds arising from the solution of the Riemann problem for the Euler equations. 
In Sect. 4 we propose three new approaches to estimate theoretical bounds for the minimal and maximal wave speeds for the Euler equations at each interface; we also perform a comparison of all present methods against available estimates through carefully devised test problems. In Sect. 5 we present wave speed bounds for the shallow water equations and assess the results. In Sect. 6 we do so for the blood flow equations for arteries. Conclusions are drawn in Sect. 7.

\section{The Euler Equations: Exact Wave Speeds}

In this section we set the scene, define the scope of the paper and provide exact wave speed reference solutions for the Euler equations. 

\subsection{Equations and the Riemann problem}

The Euler equations in one space dimension, in differential conservation form, read
\begin{equation}                            \label{ee1}
      \partial_{t} {\bf Q}(x,t) + \partial_{x} {\bf F} ({\bf Q}(x,t)) = {\bf 0}  \;. 
\end{equation}
${\bf Q}(x,t)=\left[\rho, \rho u, E \right]^{T}$ is the vector of conserved variables and ${\bf F(Q)}= \left[\rho u, \rho u^{2} + p, u(E+p)\right]^{T}$ is the flux vector. Here, $\rho$ is density, $u$ is particle velocity, $p$ is pressure and $E$ is total energy given in terms of kinetic energy $\frac{1}{2}u^{2}$ and specific internal energy $e$ as
\begin{equation}                            \label{ee2}
      E = \rho(\frac{1}{2}u^{2} + e)  \;,  \hspace{3mm} e(\rho, p) = \frac{p}{\rho(\gamma-1)} \;.
\end{equation}
The function $e = e(\rho, p)$ is called the (caloric) {\it equation of state}. Here we have taken the ideal gas case, where $\gamma$ is the ratio of specific heats, 
taken here as the constant $\gamma=1.4$. The eigenvalues of the Euler equations are
\begin{equation}                            \label{ee3}
      \lambda_{1} ({\bf Q}) = u-c \;, \hspace{3mm} \lambda_{2} ({\bf Q}) = u \;,  \hspace{3mm}\lambda_{3} ({\bf Q})= u+c \;, 
\end{equation}
where $c$ is the speed of sound,  given as
\begin{equation}                            \label{ee4}
       c = \sqrt{\frac{\gamma p}{\rho}} \;. 
\end{equation}
The corresponding right eigenvectors are
\begin{equation}                            \label{ee5}
       {\bf R}_{1}({\bf Q})=\left[\begin{array}[c]{cc}
		1 \\
		u-c \\
		H-uc
		\end{array}\right] \;,\;
        {\bf R}_{2}({\bf Q})=\left[\begin{array}[c]{cc}
		1 \\
		u \\
		\frac{1}{2}u^{2} 
		\end{array}\right] \;,\;
        {\bf R}_{3}({\bf Q})=\left[\begin{array}[c]{cc}
		1 \\
		u+c \\
		H+uc     
		\end{array}\right]  \;,
\end{equation}
where 
\begin{equation}                            \label{ee5b}
        H=(E+p)/\rho
\end{equation}
is the total specific enthalpy and scaling factors for the eigenvectors have been set to unity. For background see \cite{Toro:2009a}. \\

The Riemann problem for the Euler equations (\ref{ee1}) is the initial value problem 
\begin{equation}                                               \label{ee6}
	\left. \begin{array}{ll}
		\mbox{PDEs:} & \partial_{t}{\bf Q}(x,t) + \partial_{x} {\bf F}({\bf Q}(x,t))={\bf 0} \;,  \hspace{3mm} x \in \mathbb{R} \;, \hspace{3mm} t>0 \;,\\
		\mbox{ICs:}  & {\bf Q} (x,0) = 
			\left\{ \begin{array}{lcc}
			{\bf Q}_{L}     &  \mbox{ if } & x < 0 \;,  \\
		    {\bf Q}_{R}     &  \mbox{ if } & x > 0  \;, 
			\end{array} \right.
	\end{array} \right\}
\end{equation}
with ${\bf Q}_{L}$ and ${\bf Q}_{R}$ two prescribed constant vectors.
\begin{figure}
         \centerline{
         \includegraphics[scale=0.8,angle=0]{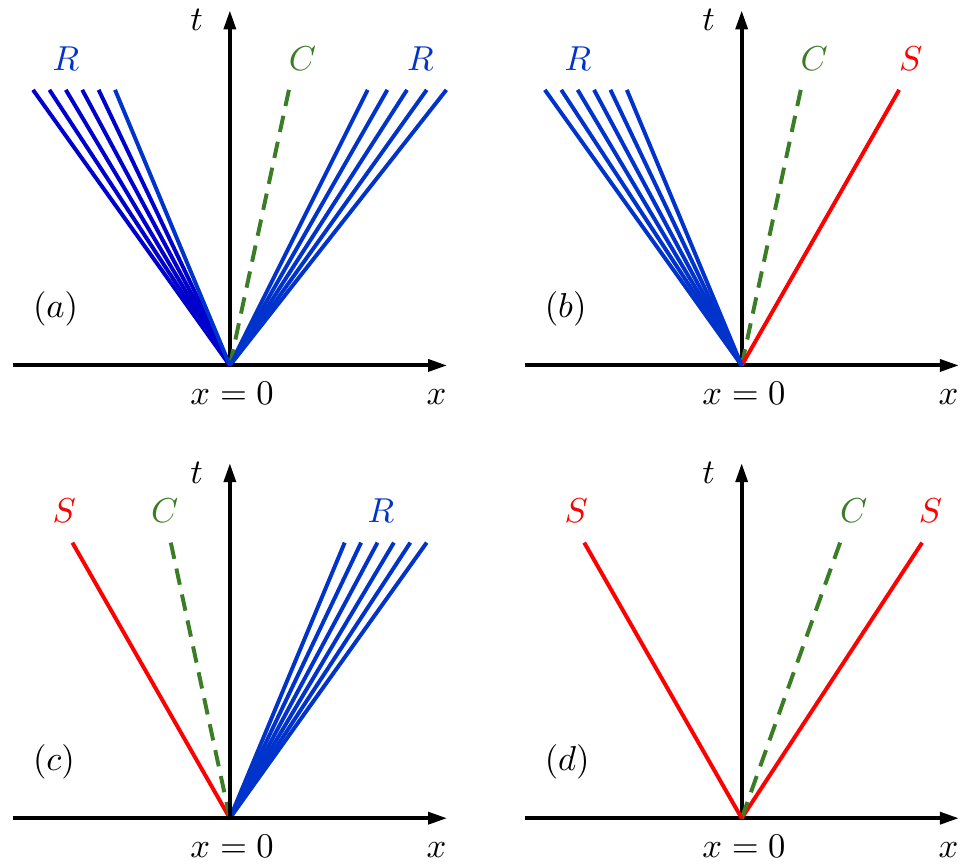}    
           }
          \vspace{2mm}
           \caption{Four possible wave patterns emerging from the  solution of the Riemann problem for the ideal, compressible Euler equations: (a) Rarefaction-Contact-Rarefaction; (b) Rarefaction-Contact-Shock; (c) Shock-Contact-Rarefaction and (d) Shock-Contact-Shock.}
           \label{fig:EEFourPatterns}
\end{figure}
The solution of (\ref{ee6}) consists of three wave families corresponding to the $\lambda_{k}-$characteristic fields,  $k=1,2,3$,  and separate four constant regions in the half $x$-$t$ plane. The resulting waves may be of three types: shocks (S), rarefactions (R) and contact discontinuities (C).
Fig. \ref{fig:EEFourPatterns} depicts the four possible wave patterns arising from the solution, namely: (a) Rarefaction-Contact-Rarefaction; (b) Rarefaction-Contact-Shock; (c) Shock-Contact-Rarefaction and (d) Shock-Contact-Shock. Fig.\ref{fig:EERPstrategy} depicts the generic structure of the solution in the $x$-$t$ plane. The solution in the entire $x$-$t$  half plane is characterised by four constant regions (wedges) $\mathcal R_{0}$, $\mathcal R_{1}$, $\mathcal R_{2}$ and $\mathcal R_{3}$ separated by three wave families. The  $\lambda_{1}$ and  $\lambda_{3}$ characteristic fields are genuinely non-linear and are associated with either shocks or rarefactions; the intermediate field associated with $\lambda_{2}$ is linearly degenerate and is associated with a contact discontinuity. Regions  $\mathcal R_{0}$ (left data) and  $\mathcal R_{3}$ (right data) are known {\it a-priori}. $\mathcal R_{1}$  and  $\mathcal R_{2}$ must be found, as well as the wave types associated to the non-linear fields (shocks or rarefactions).  

Finding the solution of the Riemann problem means finding ${\bf Q}(x,t)$ for $x \in \mathbb{R}$ and $t >0$.  Here we follow the method proposed in \cite{Toro:1989d}, see also Chap. 4 of \cite{Toro:2009a}. The solution strategy proceeds in two steps. First, one computes the pressure and velocity in the  {\it Star Region} depicted in Fig. \ref{fig:EERPstrategy}, called the {\it Star Region}. Then, the wave types are determined and the  solution in the full half plane is found. 

\subsection{Pressure in the Star Region}

Here we establish equations for computing the pressure  $p_{*}$ and the particle velocity $u_{*}$ in the {\it Star Region}, see Fig. \ref{fig:EERPstrategy}. 
Next we state a lemma that includes relevant properties of the Euler equations that will be used to discuss estimates for wave speed bounds.

 \begin{lemma} \label{eulRI} 
\noindent{\bf Exact Riemann solution and wave speeds:}  

\begin{enumerate}

\item {\bf Wave jumps across rarefactions: } Across the left rarefaction  the left Riemann invariant gives
\begin{equation}                         \label{eul:10}
      u_{*} + \frac{2}{\gamma-1} c_{*L} = u_{L} + \frac{2}{\gamma-1} c_{L}  \;,
\end{equation}
which implies the relations
\begin{equation}                         \label{eul:10b}
      u_{*} = u_{L} - f_{L}  \;; \hspace{3mm}
		f_{L}  =  \frac{2}{\gamma-1} ( c_{*L}-c_{L})  \;.
\end{equation}
Across the right rarefaction the right Riemann invariant gives
\begin{equation}                         \label{eul:11}
      u_{*} - \frac{2}{\gamma-1} c_{*R}= u_{R} - \frac{2}{\gamma-1}c_{R}  \;,
\end{equation}
which in turn implies the relations
\begin{equation}                         \label{eul:10b}
      u_{*} = u_{R} + f_{R}  \;; \hspace{3mm}
		f_{R}  =  \frac{2}{\gamma-1} ( c_{*R}-c_{R})  \;.
\end{equation}

\item  {\bf Wave jumps across shocks:} For a left shock, the Rankine-Hugoniot conditions give
\begin{equation}                          \label{eul:12} 
		u_{*} = u_{L} - f_{L}  \;; \hspace{3mm}
		f_{L}  = (p_{*}-p_{L})\sqrt{\frac{A_L}{p_* + B_L} } \;; \hspace{3mm} A_L =\frac{2}{(\gamma+1)\rho_L} \;; \hspace{3mm} B_L =\frac{\gamma-1}{\gamma+1}p_L \;.
\end{equation}
For a right shock wave the Rankine-Hugoniot conditions give
\begin{equation}                          \label{eul:14} 
		u_{*} = u_{R} + f_{R}  \;; \hspace{3mm}
		f_{R}  = (p_{*}-p_{R})\sqrt{\frac{A_R}{p_* + B_R} } \;; \hspace{3mm} A_R =\frac{2}{(\gamma+1)\rho_R}  \;; \hspace{3mm} B_R =\frac{\gamma-1}{\gamma+1}p_R \;.
\end{equation}
\item {\bf Shock speeds:} The speed of a left shock is
\begin{equation}                          \label{eul:13}
S_{L} = u_{L} - c_{L} q_{L} \;, \hspace{3mm} 
       q_{L} = \sqrt{1+\frac{\gamma+1}{2 \gamma}(y-1)} \;,  \hspace{3mm} y= \frac{p_{*}}{p_{L}} \;
\end{equation}
and of a right shock is
\begin{equation}                           \label{eul:15}
       S_{R} = u_{R} + c_{R} q_{R} \;, \hspace{3mm} 
       q_{R} = \sqrt{1+\frac{\gamma+1}{2 \gamma}(y-1)} \;,  \hspace{3mm} y= \frac{p_{*}}{p_{R}} \;.
\end{equation}
\item {\bf Solution for pressure $p_{*}$.} The solution for pressure $p_{*}$ in the Riemann problem   (\ref{ee6}) for the ideal Euler equations  (\ref{ee1})  is given by the root of 
\begin{equation}                          \label{ee7}
	  f(p,{\bf Q}_{ L},{\bf Q}_{ R}) \equiv f_{ L}(p,{\bf Q}_{ L})
	  +f_{ R}(p,{\bf Q}_{ R}) +\Delta u =0  \;, \hspace{3mm} \Delta u \equiv u_{ R}-u_{ L} \;,\;
\end{equation}
where the shock and rarefaction branches are
\begin{equation}                          \label{ee8} 
      f_{ L}(p,{\bf Q}_{ L})=\left\{
      \begin{array}{ll}
		 \displaystyle{ \left(p - p_{ L}\right)
		 \left[\frac{A_{ L}}{p + B_{ L}}\right]^{\frac{1}{2}} } \;, &  \mbox{  if $p>p_{ L}$ (shock) } \;,\;\\  &  \\
		\displaystyle{ \frac{2c_{ L}}{(\gamma -1)} \left [
		\left( \frac{p}{p_{ L}} \right) 
		^{\frac{\gamma -1}{2\gamma }}  -1 \right] }  \;,     &
		\mbox{if $p\leq p_{ L}$ (rarefaction)} \;,\;    
		\end{array} \right.
\end{equation} 
\begin{equation}                         \label{ee9} 
      f_{ R}(p,{\bf Q}_{ R})=\left\{
		 \begin{array}{ll}
		 \displaystyle{ \left(p - p_{ R}\right)
		 \left[\frac{A_{ R}}{p + B_{ R}}\right]^{\frac{1}{2}} }  \;,
		 &  \mbox{  if $p>p_{ R}$ (shock) } \;,\;\\
		 &    \\
		\displaystyle{ \frac{2c_{ R}}{(\gamma -1)} \left [
		\left( \frac{p}{p_{ R}} \right) 
		^{\frac{\gamma -1}{2\gamma }}  -1 \right]}   \;,   &
		\mbox{if $p\leq p_{ R}$ (rarefaction)} \;,\;    
		\end{array} \right.
\end{equation}
with $A_{ K}$ and  $B_{ K}$ (K=L,R) defined in (\ref{eul:12}) and  (\ref{eul:14}).

\item  {\bf Solution for velocity $u_{*}$.} Once $p_{*}$ is determined from solving (\ref{ee7}), the solution for the velocity $u_{*}$ in the Star Region follows as
\begin{equation}                        \label{ee11} 
      u_{*}=\frac{1}{2} (u_{ L}+u_{ R}) + \frac{1}{2} \left[f_{ R}(p_{*}, {\bf Q}_{ R})-f_{ L}(p_{*}, {\bf Q}_{ L})\right]  \;.
\end{equation}

\item {\bf All-rarefaction solution.} In the special case in which the function $f$ in (\ref{ee7}) is made up of rarefaction branches, then equation (\ref{ee7})  has exact solution given as
\begin{equation}                          \label{eeAllRarefaction}
      p_{ *rr}=\left[\frac{c_{ L}+c_{ R}-\frac{1}{2}(\gamma-1)(u_{ R}-u_{ L})}
      {c_{ L}/p_{ L}^{\frac{\gamma-1}{2\gamma}}
      +c_{ R}/p_{ R}^{\frac{\gamma-1}{2\gamma}}}
      \right]^{\frac{2\gamma}{\gamma-1}} \;.
\end{equation}

\end{enumerate}
\end{lemma}
\begin{proof}
 (omitted). See \cite{Toro:2009a} for details.
\end{proof}

\subsection{Further properties of the pressure function} 
\label{subsec:PropertiesPressureFunction}

Here we collect some particular properties of the pressure function $f$ in (\ref{ee7}). Fig. \ref{fig:EERPstrategy} shows that  $f_{ L}$ governs relations {\it across} the left non-linear wave and connects the  unknown particle speed $u_{*}$ to the known state ${\bf Q}_{ L}$ on the left side. Analogously,  $f_{ R}$ governs relations {\it across} the right wave and connects  $u_{*}$ to ${\bf Q}_{ R}$.  The form of $f_{ L}$ and $f_{ R}$ depends on whether the corresponding  non-linear wave is a shock or a rarefaction. For shocks one applies the Rankine-Hugoniot conditions  and for rarefactions one applies  generalised Riemann invariants. The functions $f_{ L}$ and $f_{ R}$ are called Lax curves, or simply wave curves.
\begin{figure}
         \centerline{
         \includegraphics[scale=0.8]{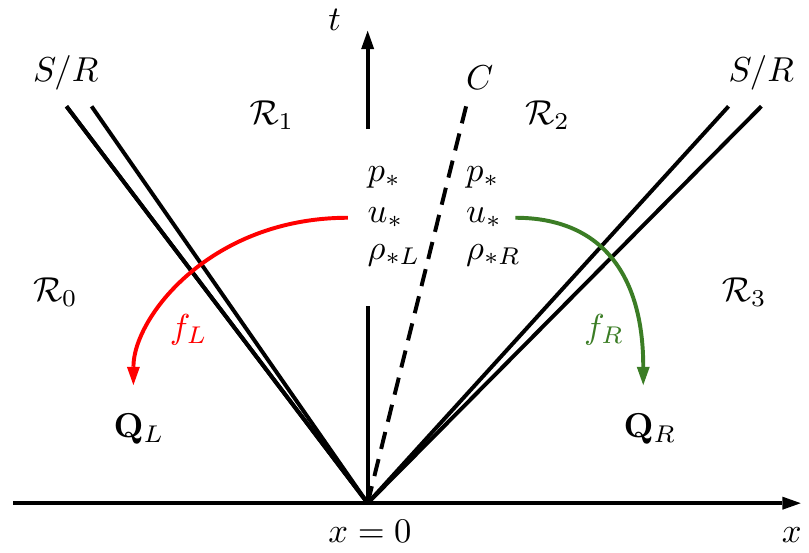}    
           }
        \caption{General wave configuration for the solution of the Riemann problem for the Euler equations. The {\it Star Region} is  connected to left and right data states via functions
        $f_{L}$ and $f_{R}$, respectively, see equations (\ref{ee8}) and (\ref{ee8}). Pressure and velocity are constant across the contact discontinuity, while density changes discontinuously.}
        \label{fig:EERPstrategy}
\end{figure}
Analysis of the pressure function $f(p)= f_{ L}+f_{ R}+\Delta u$ in (\ref{ee7}) reveals that,  for physically admissible data, the solution $p_{*}$ exists and is unique. It can be verified that $f(p)$ is monotone increasing and concave down, by calculating the first and second derivatives of $f_{ K}$ (K=L,R) with  respect to $p$.  It can also be seen that  $f'_{ K} \rightarrow 0 $ as $ p \rightarrow \infty $ and  $f''_{ K} \rightarrow 0 $ as $p \rightarrow \infty $. Fig. \ref{fig:EEPressureFunction} depicts $f(p)$ for fixed  initial data  for density and pressure $ \rho_{ L},~p_{ L} $ and $ \rho_{ R},~p_{ R} $ and for three cases of velocity difference $\Delta u=u_{ R}-u_{ L}$ denoted as $\Delta u_{1}$, $\Delta u_{2}$ and $\Delta u_{3}$. By decreasing $\Delta u$  from $\Delta u_{1}$ to $\Delta u_{3}$  $f(p)$ is shifted rightwards, giving rise to roots $p_{*1}$, $p_{*2}$ and $p_{*3}$ in increasing order.  $\Delta u$, $p_{ L}$ and $p_{ R}$ are the most important parameters for $f(p)$. See Chap. 4 in  \cite{Toro:2009a} for details.

With reference to  Fig. \ref{fig:EEPressureFunction} we define
\begin{equation}                        \label{ee12} 
      p_{ min}=\min\{p_{ L},p_{ R}\}  \;, \hspace{2mm} p_{ max}=\max\{p_{ L},p_{ R}\} \;,  \hspace{2mm}
      f_{ min}=f(p_{ min}) \;,  \hspace{2mm}  f_{ max}=f(p_{ max}) \;.              
\end{equation}   
For given $p_{ L}$, $p_{ R}$ it is the velocity difference $\Delta u $  which determines the value of $p_{*} $ and the specific wave pattern (one out of four) in Fig.  \ref{fig:EEFourPatterns}, namely
\begin{equation}                        \label{ee13} 
          \left.\begin{array}{llll}
           p_{*} \in I_{1} = (0,p_{ min})                     & \mbox{  if  }  & f_{ min} >0  \mbox{ and } f_{ max} >0          & (\mbox{R/C/R}) \;, \\
           p_{*} \in I_{2} = [p_{ min}, p_{ max}]   & \mbox{  if  }  & f_{ min} \leq 0  \mbox{ and } f_{ max} \geq0      & ( \mbox{R/C/S or S/C/R}) \;, \\
           p_{*} \in I_{3} = (p_{ max}, \infty)             & \mbox{  if  }  & f_{ min} <0  \mbox{ and } f_{ max} <0           & (\mbox{S/C/S}) \;.
          \end{array}\right\}
\end{equation}
The wave pattern can be identified {\it a priori} without solving the Riemann problem, by simply noting the signs of $f_{ min}$ and $f_{ max}$. For non-vacuum initial data ${\bf Q}_{ L}$, ${\bf Q}_{ R}$ there exists a unique positive solution $p_{*}$ for pressure,  provided $\Delta u$  satisfies the {\it pressure positivity condition} 
\begin{equation}                        \label{ee14}
      (\Delta u)_{ crit} \equiv \frac{2c_{ L}}{\gamma-1} + \frac{2c_{ R}}{\gamma-1} > u_{ R}-u_{ L}  \;.
\end{equation}
Otherwise, {\it vacuum}  is generated by the non-linear waves, usually very strong rarefactions. The structure of the solution in this case is different from that depicted in Fig. \ref{fig:EERPstrategy} and so  is the method of solution; for details see Chap. 4 of \cite{Toro:2009a}.
\begin{figure}
      \centerline{\includegraphics[scale=1.0]{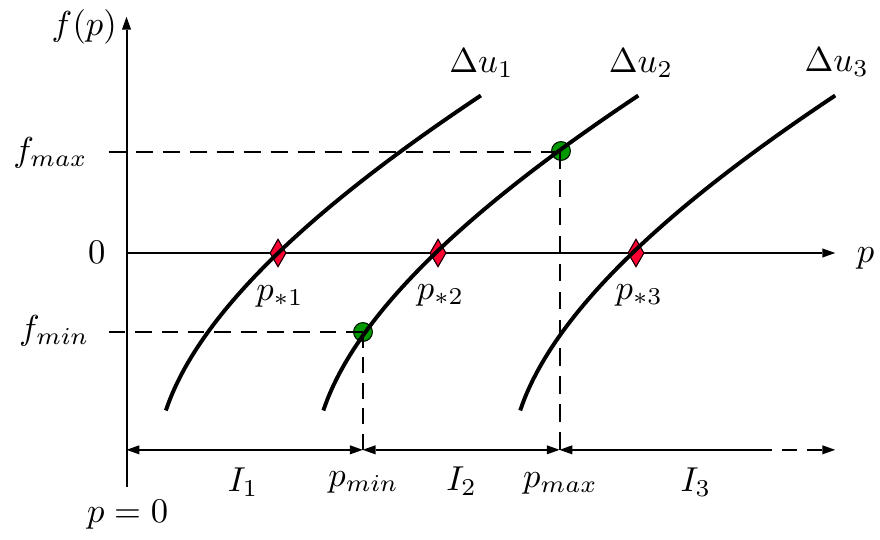}
      }
      \caption{Behaviour of the pressure function $f(p)$, see (\ref{ee7}), in the solution of the Riemann problem for the Euler equations assuming $\Delta u_{1}>\Delta u_{2}>\Delta u_{3}$.}
      \label{fig:EEPressureFunction}
\end{figure}

From Lemma \ref{eulRI}  we summarise the results for the minimal and maximal wave speeds emerging from the Riemann problem. For the (minimal) left wave speed we have
\begin{equation}                          \label{EEleftwave}
\begin{array}{c}
S_{L}^{Ex}= \left\{ 
       \begin{array}{lccc} 
       u_{L} - c_{L}   & \mbox{if} &  p_{*} \le p_{L}  & \mbox{left rarefaction} \;, \\
       \\
       u_{L} - c_{L}q_{L}  \;, \hspace{2mm} q_{L}=\sqrt{1+\frac{\gamma+1}{2 \gamma}(\frac{p_{*}}{p_{L}}-1) }  & \mbox{if} &  p_{*} > p_{L}  & \mbox{left shock} \;.
       \end{array}\right.
\end{array}
\end{equation}
For the (maximal) right wave speed we have
\begin{equation}                          \label{EErightwave}
\begin{array}{c}
S_{R}^{Ex}= \left\{ 
       \begin{array}{lccc} 
       u_{R} + c_{R}   & \mbox{if} &  p_{*} \le p_{R}  & \mbox{right rarefaction} \;, \\
       \\
       u_{R} + c_{R}q_{R} \;, \hspace{2mm}  q_{R}= \sqrt{1+\frac{\gamma+1}{2 \gamma}(\frac{p_{*}}{p_{R}}-1) }  & \mbox{if} &  p_{*} > p_{R}  & \mbox{right shock} \;.
       \end{array}\right.
\end{array}
\end{equation}
These exact wave speeds will be used to asses the correctnes and accuracy of proposed estimates for speed bounds.  The exact expressions will also be used to produce estimates based on approximations to (\ref{EEleftwave}) and (\ref{EErightwave}).

\section{Wave Speed Estimates for the Euler Equations}
\label{sec:boundsEulerequations}

\subsection{Definition of the problem}

The core subject of this paper is represented in Fig. \ref{fig:SpeedBounds} that describes the structure of the Riemann problem solution for a typical $3 \times 3$ hyperbolic system, such as the Euler equations. It is assumed that there are three wave families associated to three distinct eigenvalues $\lambda_{1}$, $\lambda_{2}$, $\lambda_{3}$. The problem is to find {\it a priori} estimates $S_{L}^{b}$ and $S_{R}^{b}$  for the minimal (left) and maximal (right) wave speeds present in the solution of the Riemann problem for general initial conditions ${\bf Q}_{L}$ and  ${\bf Q}_{R}$.  The determination of the true minimal and maximal wave speeds $S_{L}^{Ex}$ and $S_{R}^{Ex}$ obviously depends on the actual solution of the Riemann problem. Our aim is to find the bounds $S_{L}^{b}$ and $S_{R}^{b}$ in terms of the initial conditions but  using {\bf at best none} or {\bf at worse minimal} information from the actual solution of the Riemann problem.

There are at present several ways for estimating the wave speeds  $S_{L}^{Ex}$ and $S_{R}^{Ex}$,  but ways for estimating bounds $S_{L}^{b}$ and $S_{R}^{b}$  for these are lacking. To our knowledge, the work of Guermond and Popov \cite{Guermond:2016a} is the only one available on this subject. Next we review available methods to estimate $S_{L}$ and $S_{R}$.

\subsection{Existing approaches}

There are several approaches in the current literature to the estimation of minimal and maximal wave speeds, as briefly reviewed below.\\

\noindent{\bf Davis' estimates.} Davis \cite{Davis:1988a} is credited with the following estimates for $S_{L}$ and $S_{R}$ for the one-dimensional Euler equations. The simplest one relies on evaluation of the eigenvalues $\lambda_{1}=u-c$ and $\lambda_{3}=u+c$ on the data left and right respectively. That is
\begin{equation}           \label{Davis:a}
       S^{Dav_a}_{L}= u_{L} - c_{L} \;, \qquad S^{Dav_a}_{R}= u_{R} + c_{R} \;.
\end{equation}
These estimates are the simplest ones available, which unfortunately do not constitute bounds for the true wave speeds, in general.  We remark that there are some variations in the literature; for example the particle velocity is replaced by its absolute value. We note that (\ref{Davis:a}) is actually correct for the special case in which both non-linear waves are rarefaction waves. But even for this special case the applicability is limited, as one does not generally know in advance the particular wave pattern emerging from the solution of the Riemann  problem.

Another suggestion from Davis adds some complexity but represents an improvement, relative to (\ref{Davis:a}), namely
\begin{equation}           \label{Davis:b}
         S^{Dav_b}_{L}= \min\{u_{L} - c_{L} , u_{R} - c_{R}\} \;, \qquad S^{Dav_b}_{R}= \max\{u_{L} + c_{L} , u_{R} + c_{R}\}\;.
\end{equation}
These estimates again rely on evaluation of the eigenvalues  $\lambda_{1}=u-c$ and $\lambda_{3}=u+c$, but this time such evaluation is performed on each data states, for both eigenvalues. As we shall see later, as the result of our analysis, estimates (\ref{Davis:b}) are more successful than estimates (\ref{Davis:a}) but still fail to be bounds for all possible wave configurations. \\

\noindent{\bf Einfeldt's estimates.}  Einfeldt \cite{Einfeldt:1988a} proposed wave speed estimates by incorporating information on the Riemann problem solution, via the {\it Roe averages} in the Roe Riemann solver for the Euler equations. Einfeldt proposed
\begin{equation}
       S^{Einf}_L= \widetilde{u} - \widetilde{d} \;,  \qquad S^{Einf}_R= \widetilde{u} + \widetilde{d}   \;,
\end{equation}
where 
\begin{equation} \label{EQ:utilde}
      \widetilde{u}= \frac{\sqrt{\rho_L}u_L + \sqrt{\rho_R}u_R}{\sqrt{\rho_L} + \sqrt{\rho_R}} \qquad 
\end{equation}
and
\begin{equation} \label{EQ:eta2}
      {\widetilde{d^2}} =\frac{\sqrt{\rho_L}c^2_L + \sqrt{\rho_R}c^2_R}{\sqrt{\rho_L} + \sqrt{\rho_R}} +\frac{1}{2} \frac{\sqrt{\rho_L} \sqrt{\rho_R}}  
      {\left(\sqrt{\rho_L} + \sqrt{\rho_R}\right)^2} \left(u_R-u_L\right)^2 \;.
\end{equation}

\noindent{\bf Toro's estimates.}  Toro et al. \cite{Toro:1994c} proposed estimates on the exact expressions  (\ref{EEleftwave})-(\ref{EErightwave}), namely
\begin{equation}                          \label{Toroa}
    S_{L}^{To} = u_{L} - c_{L}q_{L}    \;, \hspace{3mm}  S_{R}^{To} = u_{R} - c_{R}q_{R} \;,
\end{equation}
but the functions $q_{L}$ and  $q_{R}$ contain approximate information from the solution of the Riemann problem and can distinguish between shocks and rarefaction waves. 
The most successful choice is
\begin{equation}                          \label{Torob}
\begin{array}{c}
q_{L}=\left\{
       \begin{array}{cccc} 
       1  & \mbox{if} &  p_{*rr} \le p_{L}  & \mbox{left rarefaction} \;, \\
       \\
        \sqrt{1+\frac{\gamma+1}{2 \gamma}(\frac{p_{*rr}}{p_{L}}-1) }  & \mbox{if} &  p_{*rr} > p_{L}  & \mbox{left shock} 
       \end{array}\right.
\end{array}
\end{equation}
and
\begin{equation}                          \label{Toroc}
\begin{array}{c}
q_{R}=\left\{
       \begin{array}{cccc} 
       1  & \mbox{if} &  p_{*rr} \le p_{R}  & \mbox{right rarefaction} \;, \\
       \\
        \sqrt{1+\frac{\gamma+1}{2 \gamma}(\frac{p_{*rr}}{p_{R}}-1) }  & \mbox{if} &  p_{*rr} > p_{R}  & \mbox{right shock} \;,
       \end{array}\right.
\end{array}
\end{equation}
where $p_{*rr}$ is the {\it two-rarefaction solution} given in (\ref{eeAllRarefaction}).\\

\noindent{\bf Corollary to  Guermond-Popov lemma.} It turns out that the wave speeds (\ref{Toroa})-(\ref{Toroc}) constitute a bound for the minimal and maximal wave speeds for the Euler equations. 
This follows from Lemma 4.2  of Guermond and Popov in \cite{Guermond:2016a} which states that $p_{*rr}$ is a bound for $p_{*}$, namely
\begin{equation}                          \label{Toroe}
       p_{*rr} \ge p_{*} \;.
\end{equation}
As a matter of fact in \cite{Toro:1994a} it was conjectured that $p_{*rr}$ in general was a bound for $p_{*}$ and numerical experiments actually confirmed this,  but a rigorous proof was lacking.  Moreover, it is easily seen that the shock branches  of the functions  $q_{L}$ and  $q_{R}$ are monotone increasing functions of their pressure argument.
It follows that (\ref{Toroa})-(\ref{Toroc}) provide bounds for the minimal and maximal wave speeds in the Euler equations.\\

\noindent{\bf Batten's estimates.}  Batten et al.  \cite{Batten:1997a} also suggested to use the Roe averages to obtain wave speed estimates for the Euler equations. They also applied the idea to the three-dimensional case. Batten et al. suggested
\begin{equation}                         \label{Battena}
     S^{Ba}_L= \min\{u_{L} - c_{L} , \widetilde{u} - \widetilde{c} \} \;, \qquad S^{Ba}_R= \max\{u_{R} + c_{R} , \widetilde{u} + \widetilde{c} \}  \;,
\end{equation}
where
\begin{equation}                         \label{Battenb}
     \widetilde{u}= \frac{\sqrt{\rho_L}u_L + \sqrt{\rho_R}u_R}{\sqrt{\rho_L} + \sqrt{\rho_R}} \;, \qquad 
     \widetilde{H}= \frac{\sqrt{\rho_L}H_L + \sqrt{\rho_R}H_R}{\sqrt{\rho_L} + \sqrt{\rho_R}} \;, \qquad 
     \widetilde{c}= \left[(\gamma -1)\left(\widetilde{H}-\frac{1}{2}\widetilde{u}^2\right)\right]^{1/2} \;.
\end{equation}
$H_{K}$ is the specific enthalpy on the left and right, namely
\begin{equation}                         \label{Battenc}
        H_{K} = \frac{E_{K}+p_{K}}{\rho_{K} } \quad(K=L,R) \;.
\end{equation}

As noted in \cite{Toro:1992a}, instead of the Roe averages, one could use the solution of the Riemann problem for the linearised system in which the coefficient matrix is a frozen matrix evaluated at the Roe averages. It would be reasonable to expect a pressure $p_{*b}$ that could be used in Toro's estimates (\ref{Toroa}). Unfortunately, as reported in \cite{Toro:1992a} through a counter example, the resulting solution for pressure is not a bound for the exact solution for $p_{*}$. Therefore, using such estimate for pressure in (\ref{Torob})-(\ref{Toroc}) would not result in a bound for the wave speeds (\ref{Toroa}).\\

\noindent{\bf Guermond-Popov estimates.} Recently, Guermond and Popov  \cite{Guermond:2016a} proposed estimates for the maximal wave speed for the Euler equations with covolume equation of state. Their estimates actually bound the maximal wave speed. This is probably the first work in which theoretical bounds are put forward. They proposed a direct bound (non-iterative) and a more accurate, iterative bound. Full details are found in \cite{Guermond:2016a}. The Guermond-Popov  \cite{Guermond:2016a} estimates will be denoted as $S^{GP}_L$ and $S^{GP}_R$. \\

In the next section we propose new direct estimates for bounds of the waves speeds emerging from the solution of the Riemann problem for the Euler equations.

\begin{figure}
      \centerline{\includegraphics[scale=0.8]{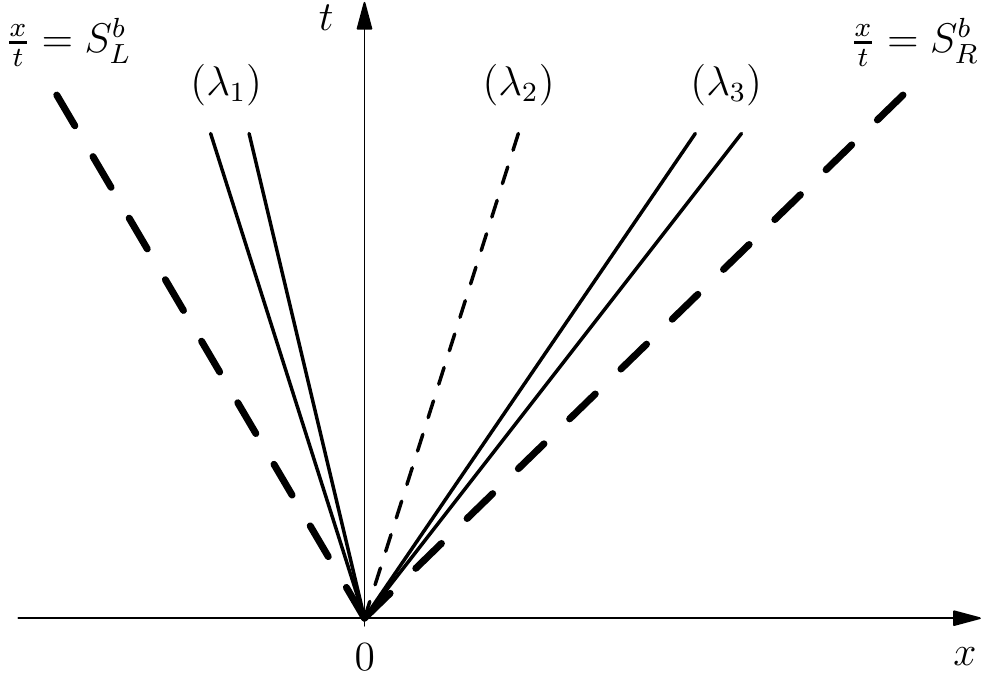}
      }
      \caption{Speed bounds $S_{L}^{b}$ and $S_{R}^{b}$ for the minimal and maximal wave speeds emerging from the solution of the Riemann problem for general initial conditions  ${\bf Q}_{L}$ and  ${\bf Q}_{R}$. }
      \label{fig:SpeedBounds}
\end{figure}

\section{Bounds for Minimal and Maximal Wave Speeds: the Euler Equations}

The main purpose of this section is to find theoretical estimates $S_{L}^{b}$ and $S_{R}^{b}$ that bound from below and from above the exact waves speed $S_{L}^{Ex}$ and $S_{R}^{Ex}$ given above, that is
 \begin{equation}                           \label{eul:16} 
 		S_{L}^{b} \le  S_{L}^{Ex} \hspace{4mm}   \mbox{and}  \hspace{4mm} S_{R}^{b} \ge  S_{R}^{Ex}  \;.
 \end{equation}

\subsection{Speed bound estimates: approach $TMS_a$}

In this approach the wave speed bounds are constructed on the bases of limited information from the solution of the Riemann problem and represents an improvement on the Toro's 
estimates (\ref{Toroa})-(\ref{Toroc}) \cite{Toro:1994c}. 

\begin{theorem} \label{eeTMSa} 

\noindent{\bf Method $TMS_a$ for the Euler equations.}  Given
\begin{equation}                                                 \label{eeTMSa1} 
p_{min} = \min\{p_{L}, p_{R} \}  \;,  \hspace{4mm} p_{max} = \max\{p_{L}, p_{R} \} \;
\end{equation}
and the three points
\begin{equation}                                                 \label{eeTMSa2} 
    P_{m}=(p_{min}, f(p_{min})) \;, \hspace{3mm} P_{M}=(p_{max}, f(p_{max})) \;,  \hspace{3mm} P_{rr}=(p_{*rr}, f(p_{*rr})) \;,
\end{equation}
where  $f(p)$ is the pressure function in (\ref{ee7}),  then the speed bounds $S_{L}^{b}$ and $S_{R}^{b}$ for the left and right waves respectively are given in the following four cases:

\begin{itemize}

\item{\bf Case R/R: two rarefaction waves.} This case is determined by the condition $f(p_{min})\ge 0$. The estimated wave speed bounds are actually exact and given as
\begin{equation}                                                \label{eeTMSa3} 
		 S_{L}^{b} = u_{L} -  c_{L} \;, \hspace{4mm} S_{R}^{b}  = u_{R} +  c_{R}  \;.
\end{equation}
\item{\bf Case R/S: left rarefaction/right shock}. This case is determined by the conditions $f(p_{min}) < 0$, $f(p_{max}) > 0$ and $p_{min} = p_{R}$, The value
\begin{equation}                                                \label{eeTMSa4} 
		  p_{*mM} = p_{min}-\left[\frac{p_{max}-p_{min} }{ f(p_{max})-f(p_{min})}\right]  f(p_{min})  \;
\end{equation}
is obtained from linear interpolation based on points $P_{m}$ and $P_{M}$, so that the estimated wave speed bounds are
\begin{equation}                                                \label{eeTMSa5} 
		 S_{L}^{b} = u_{L} -  c_{L} \;, \hspace{4mm} S_{R}^{b}  = u_{R} +  c_{R}q_{R}(p_{*mM})  \;, \hspace{4mm} q_{R}(\hat p) =\sqrt{1+\frac{\gamma+1}{2 \gamma}(\frac{\hat p}{p_{R}}-1) } \;,
\end{equation}
\item{\bf Case S/R: left shock/right rarefaction.} The conditions are: $f(p_{min}) < 0$, $f(p_{max}) > 0$ and $p_{min} = p_{L}$. The estimated wave speed bounds are 
\begin{equation}                                                \label{eeTMSa6} 
		 S_{L}^{b} = u_{L} -  c_{L}q_{L}(p_{*mM}) \;, \hspace{4mm} S_{R}^{b}  = u_{R} +  c_{R}  ;, \hspace{4mm} q_{L}(\hat p) =\sqrt{1+\frac{\gamma+1}{2 \gamma}(\frac{\hat p}{p_{L}}-1) } \;.
\end{equation}
\item{\bf Case S/S: two shock waves.} The condition is $f(p_{max}) < 0$. The estimated wave speed bounds are 
\begin{equation}                                                \label{eeTMSa7} 
		 S_{L}^{b} = u_{L} -  c_{L}q_{L}(p_{*Mrr}) \;, \hspace{4mm} S_{R}^{b}  = u_{R} +  c_{R} q_{R}(p_{*Mrr}) \;.
\end{equation}
where
\begin{equation}                                                \label{eeTMSa8} 
		  p_{*Mrr} = p_{max}-\left[\frac{p_{*rr}-p_{max} }{ f(p_{*rr})-f(p_{max})}\right]  f(p_{max}) \;
\end{equation}
is obtained from linear interpolation based on points $P_{M}$ and $P_{rr}$.

\end{itemize}
Table \ref{tab:TMSaBounds} summarises all cases.\\
\end{theorem}
\begin{proof}
   The justification of the method follows from the properties of the functions $f(p)$, $q_{K}(p)$ ($K=L,R$) and Lemma 4.2 of Guermond and Popov  \cite{Guermond:2016a};  see (\ref{Toroe}).
\end{proof}
\begin{table}
\begin{center}
\begin{tabular}{|c|c|c|c|} \hline
Wave pattern   &   Conditions                                                                          &  $S_{L}^{b}$                                    &   $S_{R}^{b}$                                 \\ \hline
R/R                  &  $f(p_{min})\ge 0$                                                                 &  $u_{L} - c_{L}$                         &   $u_{R} + c_{R}$                          \\ \hline 
R/S                  &  $f(p_{min}) < 0$ \;, $f(p_{max}) > 0$ \;, $p_{min}=p_{R}$   &  $u_{L} -  c_{L}$                        &   $u_{R} +  c_{R}q_{R}(p_{*mM}) $ \\ \hline
S/R                  &  $f(p_{min}) < 0$ \;, $f(p_{max}) > 0$ \;, $p_{min}=p_{L}$   & $u_{L} -  c_{L}q_{R}(p_{*mM})$   &   $u_{R} +  c_{R}$                         \\ \hline
S/S                  &  $f(p_{max}) < 0$                                                                  &  $u_{L} -  c_{L}q_{L}(p_{*Mrr})$  &   $u_{R} + c_{R}q_{R}(p_{*Mrr})$   \\ \hline 
\end{tabular}     
\end{center}
\begin{center}
\caption{$TMS_a$ estimates for bounds $S_{L}^{b}$ and $S_{R}^{b}$ on minimal and maximal wave speeds. Function $q_{K}(p)$ ($K=L,R$) is given in (\ref{eul:13}) and (\ref{eul:15}).}\label{tab:TMSaBounds}  
\end{center}
\end{table}
%


\subsection{Speed bound estimates: approach $TMS_b$}

In this approach the wave speed bounds are constructed on the bases of limited information from the solution of the Riemann problem and is a simplified version of approach $TMS_a$ seen previously in Theorem \ref{eeTMSa}.

\begin{theorem} \label{eeTMSb} 

\noindent{\bf Method $TMS_b$ for the Euler equations.}  
Given
\begin{equation}                                                 \label{eeTMSb1} 
p_{min} = \min\{p_{L}, p_{R} \}  \;,  \hspace{4mm} p_{max} = \max\{p_{L}, p_{R} \} \;
\end{equation}
and the two points
\begin{equation}                                                 \label{eeTMSb2} 
    P_{m}=(p_{min}, f(p_{min})) \;, \hspace{3mm}   \hspace{3mm} P_{rr}=(p_{*rr}, f(p_{*rr})) \;,
\end{equation}
where  $f(p)$ is the pressure function in (\ref{ee7}),  then the speed bounds $S_{L}^{b}$ and $S_{R}^{b}$ for the left and right waves respectively given in the following four cases:

\begin{itemize}

\item{\bf Case R/R: two rarefaction waves.} This case is determined by the condition $f(p_{min})\ge 0$. The estimated wave speed bounds are actually exact and given as
\begin{equation}                                                \label{eeTMSb3} 
		 S_{L}^{b} = u_{L} -  c_{L} \;, \hspace{4mm} S_{R}^{b}  = u_{R} +  c_{R}  \;.
\end{equation}
\item{\bf Case R/S: left rarefaction and right shock. } The conditions are: $f(p_{min}) < 0$, $f(p_{max}) > 0$ and $p_{min}=p_{R}$. The estimated wave speed bounds are given as
\begin{equation}                                                \label{eeTMSb4} 
		 S_{L}^{b} = u_{L} -  c_{L} \;, \hspace{4mm} S_{R}^{b}  = u_{R} +  c_{R}q_{R}(p_{*mrr})  \;, \hspace{4mm} q_{R}(\hat p) =\sqrt{1+\frac{\gamma+1}{2 \gamma}(\frac{\hat p}{p_{R}}-1) } \;,
\end{equation}
where $p_{*mrr}$  is obtained from linear interpolation based on points $P_{m}$ and $P_{rr}$, namely  
\begin{equation}                                                \label{eeTMSb5} 
		  p_{*mrr} = p_{min}-\left[\frac{p_{*rr}-p_{min} }{ f(p_{*rr})-f(p_{min})}\right]  f(p_{min})  \;.
\end{equation}
\item{\bf Case S/R: left shock/right rarefaction. } The conditions are: $f(p_{min}) < 0$, $f(p_{max}) > 0$ and $p_{min}=p_{L}$. The estimated wave speed bounds are given as
\begin{equation}                                                \label{eeTMSb6} 
		 S_{L}^{b} = u_{L} -  c_{L}q_{L}(p_{*mrr}) \;, \hspace{4mm} S_{R}^{b}  = u_{R} +  c_{R}  ;, \hspace{4mm} q_{L}(\hat p) =\sqrt{1+\frac{\gamma+1}{2 \gamma}(\frac{\hat p}{p_{R}}-1) } \;.
\end{equation}
\item{\bf Case S/S: two shock waves.} The condition is: $f(p_{max}) < 0$. The estimated wave speed bounds are given as
\begin{equation}                                                \label{eeTMSb7} 
		 S_{L}^{b} = u_{L} -  c_{L}q_{L}(p_{*mrr}) \;, \hspace{4mm} S_{R}^{b}  = u_{R} +  c_{R} q_{R}(p_{*mrr}) \;.
\end{equation}
\end{itemize}
Table \ref{tab:TMSbBounds} summarises all cases.\\
\end{theorem}
\begin{proof}
  The proof is similar to that of  Theorem \ref{eeTMSa} . Details are omitted.
\end{proof}

\begin{table}
\begin{center}
\begin{tabular}{|c|c|c|c|} \hline
Wave pattern   &   Conditions                                                                           &  $S_{L}^{b}$                             &   $S_{R}^{b}$                                 \\ \hline
R/R                  &  $f(p_{min})\ge 0$                                                                 &  $u_{L} - c_{L}$                         &   $u_{R} + c_{R}$                          \\ \hline 
R/S                  &  $f(p_{min}) < 0$ \;, $f(p_{max}) > 0$ \;, $p_{min}=p_{R}$   &  $u_{L} - c_{L}$                        &   $u_{R} +  c_{R}q_{R}(p_{*mrr}) $ \\ \hline
S/R                  &  $f(p_{min}) < 0$ \;, $f(p_{max}) > 0$ \;, $p_{min}=p_{L}$   & $u_{L} -  c_{L}q_{R}(p_{*mrr})$   &   $u_{R} + c_{R}$                         \\ \hline
S/S                  &  $f(p_{max}) < 0$                                                                  &  $u_{L} - c_{L}q_{L}(p_{*mrr})$  &   $u_{R} + c_{R}q_{R}(p_{*mrr})$   \\ \hline 
\end{tabular}       
\end{center}
\begin{center}
\caption{$TMS_b$ estimates for bounds $S_{L}^{b}$ and $S_{R}^{b}$ on minimal and maximal wave speeds.}
\label{tab:TMSbBounds}
\end{center}
\end{table}
%


\subsection{Speed bound estimates: approach $TMS_c$}

This approach to estimate wave speed bounds is inspired by estimates (\ref{Toroa})-(\ref{Toroc}) \cite{Toro:1994c}. The estimates are constructed on the bases of limited information from the wave patterns present in the solution of the Riemann problem. We now identify all possible cases, along with the respective bound estimates:\\

\begin{theorem}

\noindent{\bf Method $TMS_c$ for the Euler equations.}   Given

\begin{equation}                                                 \label{eeTMSc1} 
p_{min} = \min\{p_{L}, p_{R} \}  \;,  \hspace{4mm} p_{max} = \max\{p_{L}, p_{R} \} \;,
\end{equation}
then the speed bounds $S_{L}^{b}$ and $S_{R}^{b}$ for the left and right waves respectively are given in the following four cases:

\begin{itemize}

\item {\bf Case R/R: two rarefaction waves.} This case is determined by the condition $f(p_{min})\ge 0$, where $f(p)$ is the pressure function in (\ref{ee7}). Then estimated wave speed bounds are actually exact and are
\begin{equation}                                                \label{eeTMSc2} 
		 S_{L}^{b} = u_{L} -  c_{L} \;, \hspace{4mm} S_{R}^{b}  = u_{R} +  c_{R}  \;.
\end{equation}
\item {\bf Case R/S: left rarefaction/right shock.} The conditions that identify this case are: $f(p_{min}) < 0$, $f(p_{max}) > 0$ and $p_{max}=p_{R}$. The estimated wave speed bounds are 
\begin{equation}                                                \label{eeTMSc3} 
		 S_{L}^{b} = u_{L} -  c_{L} \;, \hspace{4mm} S_{R}^{b}  = u_{R} +  c_{R}q_{R}(p_{L})  \;, \hspace{4mm} q_{R}(\hat p) =\sqrt{1+\frac{\gamma+1}{2 \gamma}(\frac{\hat p}{p_{R}}-1) } \;.
\end{equation}
\item {\bf Case S/R: left shock/right rarefaction.} The conditions that identify this case are: $f(p_{min}) < 0$,  $f(p_{max}) > 0$ and $p_{max}=p_{R}$. The estimated wave speed bounds are 
\begin{equation}                                                \label{eeTMSc4} 
		 S_{L}^{b} = u_{L} -  c_{L}q_{L}(p_{R}) \;, \hspace{4mm} S_{R}^{b}  = u_{R} +  c_{R}  ;, \hspace{4mm} q_{L}(\hat p) =\sqrt{1+\frac{\gamma+1}{2 \gamma}(\frac{\hat p}{p_{L}}-1) } \;.
\end{equation}
\item{\bf Case S/S: left shock/right  shock.} The conditions that identify this case is: $f(p_{max}) < 0$ (shock/shock). The estimated wave speed bounds are 
\begin{equation}                                                \label{eeTMSc5} 
		 S_{L}^{b} = u_{L} -  c_{L}q_{L}(p_{*rr}) \;, \hspace{4mm} S_{R}^{b}  = u_{R} +  c_{R} q_{R}(p_{*rr}) \;,
\end{equation}
where 
\begin{equation}                                               \label{eeTMSc6}      
       p_{*rr}=\left[\frac{c_{ L}+c_{ R}-\frac{1}{2}(\gamma-1)(u_{ R}-u_{ L})}
      {c_{ L}/p_{ L}^{\frac{\gamma-1}{2\gamma}}
      +c_{ R}/p_{ R}^{\frac{\gamma-1}{2\gamma}}}
      \right]^{\frac{2\gamma}{\gamma-1}} \;.
\end{equation}
\end{itemize}
Table \ref{tab:TMScBounds} summarises all cases.\\
\end{theorem}
\begin{proof}
         The proof is based on the monotone increasing character of the functions $q_{L}(\hat p)$  and $q_{R}(\hat p)$. 
\end{proof}

\begin{table}
\begin{center}
\begin{tabular}{|c|c|c|c|} \hline
Wave pattern       &   Conditions                                                                           &   $S_{L}^{b}$                         &   $S_{R}^{b}$  \\ \hline
$R/R$                  &  $f(p_{min})\ge 0$                                                                  &  $u_{L} - c_{L}$                     &   $u_{R} + c_{R}  $  \\ \hline 
$R/S$                  &  $f(p_{min}) < 0$ \;, $f(p_{max}) > 0$ \;, $p_{min}=p_{R}$    &  $u_{L} -  c_{L}$                    &  $ u_{R} +  c_{R}q_{R}(p_{L}) $ \\ \hline
$S/R$                  &  $f(p_{min}) < 0$ \;, $f(p_{max}) > 0$ \;, $p_{min}=p_{L}$    &  $ u_{L} -  c_{L}q_{L}(p_{R}) $  & $  u_{R} +  c_{R} $ \\ \hline
$S/S$                  &  $f(p_{max}) < 0$                                                                   &  $  u_{L} -  c_{L}q_{L}(p_{*rr})   $           &$   u_{R} + c_{R}q_{R}(p_{*rr})  $  \\ \hline 
\end{tabular}       
\caption{$TMS_c$ estimates for bounds $S_{L}^{b}$ and $S_{R}^{b}$ on minimal and maximal wave speeds. Function $q_{K}(p)$ ($K=L, R$) given by 
Eqs.  (\ref{eeTMSb4})-(\ref{eeTMSb5}). }
\label{tab:TMScBounds}
\end{center}
\end{table}

\noindent{\bf Remarks:} 
\begin{itemize}

\item As seen so far, for the Euler equations our approach $TMS_c$ uses the Toro's speed bound estimate for one of the four cases, namely the S/S wave pattern. We have found it difficult to derive
simple conditions for the S/S case exclusively in terms of the initial data.

\item For the shallow water equations and the blood flow equations for arterial flow, however, it is possible to construct such simple conditions in terms  of the intial data. Therefore, no 
explicit use of the Riemann problem solution is necessary in the $TMS_c$ when applied to these equations, as will be seen in Sects. 5 and 6.

\end{itemize}

\subsection{Numerical tests for the Euler equations}

In this section we perform some numerical experiments to test all the wave speed estimates considered in this paper, existing ones and the newly proposed ones.
To this end we considered seven Riemann problems. Table \ref{Tab:Euler_Tests_IC} shows the initial conditions in terms of primitive variables. The chosen tests cover all possible wave patterns, which are displayed in the last column. The exact solution for pressure $p_{*}$ and velocity $u_{*}$ in the {\it Start Region} are also displayed. The ratio of specific heats is $\gamma$ = 1.4.
\begin{table}[h!]    
\begin{center}
\begin{tabular}{|c|c|c|c|c|c|c|c|c|c|} \hline
Test & $\rho_{L}$ & $u_{L}$  & $p_{L}$ & $\rho_{R}$ & $u_{R}$  & $p_{R}$ & $u_*$& $p_*$ & wave pattern  \\ \hline
1    &  1.0       & 0.0      & 1.0     &  1.0       &  0.0     & 0.1 &  0.5248    & 0.5219 & rar-shock     \\ \hline
2    &  1.0       & 0.0      & 1.0     &  0.125     &  0.0     & 0.1 &  0.9274 & 0.3031   & rar-shock \\ \hline
3    &  1.0       & 0.0      & 1.0     &  0.001     &  0.0     & 0.8 &  0.1794 & 0.8060    & rar-shock \\ \hline
4    &  1.0       & 0.0      & 0.01    &  1.0       &  0.0     & 1000.0 &  -19.5975 & 460.8938   & shock-rar \\ \hline
5    &  6.0       & 8.0      & 460.0   & 6.0        & -6.0     & 46.0&  3.8194  & 790.2928   & shock-shock \\ \hline 
6    &  600.0     & 80.0     & 4600.0  & 6.0        & -6.0     & 46.0&  44992.5781  & 790.2928   & shock-shock \\ \hline 
7    &  1.0       & -2.0     & 0.4     &  1.0       &  2.0     & 0.4 & 0.0000 &  0.0019    & rar-rar \\ \hline
\end{tabular}     
\caption{Initial conditions for seven  Riemann problems for the Euler equations with $\gamma = 1.4$. The exact solution for pressure $p_{*}$ and velocity $u_{*}$ in the {\it Start Region} is displayed in the $8th$ and  $9th$ columns respectively. Last column shows the emerging wave patterns. Units: SI.}
\label{Tab:Euler_Tests_IC}     
\end{center}
\end{table}

Tables \ref{Tab:Euler_Tests_SR} and \ref{Tab:Euler_Tests_SL} show the results from all seven test problems.  Comparison is made between  all wave speed estimates considered in this paper and the exact solution. Table \ref{Tab:Euler_Tests_SR} shows results for the maximal wave speed, while Table \ref{Tab:Euler_Tests_SL} shows results for the minimal wave speed. In each case the second column shows the exact solution for the respective wave speed. Columns 3 to 8 show existing estimates, while colums 9 to 11 show the new estimates proposed in this paper. All three new estimates are seen to confirm that they constitute bounds for the maximal and miminimal wave speeds. Of the existing estimates, only those of Toro ($S^{To}_{R}$, $S^{To}_{L}$) and Guermond-Popov ($S^{GP}_{R}$, $S^{GP}_{L}$) constitute bounds, while all the remaining ones are seen to fail, that is to say they do not constitute bounds for the maximal and minimal wave speeds for all cases.\\

Bounding the extreme waves is the primary objective, but accuracy is also important. It is seen that  estimates of the type $To$, $GP$, $TMS_a$, $TMS_b$, $TMS_c$ are generally accurate. Our estimates $TMS_a$ and $TMS_b$ are the most accuracte for  very strong shocks, see results for Test 6, even though the error is still large for this test problem. The simple $TMS_c$ estimates are less accurate than those of the type $TMS_a$ and $TMS_b$, in the presence of strong shocks; note also that the present $TMS_b$ estimate is more accurate than the existing $To$ and $GP$ type estimates.

\begin{table}[h!]
\scriptsize
\begin{center}
\begin{tabular}{|c|c|c|c|c|c|c|c|c|c|c|} \hline
Test & $S^{Ex}_{R}$  & $S^{Dav_a}_{R}$ & $S^{Dav_b}_{R}$   & $S^{To}_{R}$ & $S^{GP}_{R}$  & $S^{Batten}_R$ &  $S^{Einf}_{R}$  & ${S}^{TMS_a}_{R}$  & ${S}^{TMS_b}_{R}$ & ${S}^{TMS_c}_{R}$ \\ \hline
1  & 0.8039  & \textcolor{red}{\bf 0.3742} &          1.1832 &         0.8134 &         0.8134 &         0.8775 &         0.8775 &         0.9296 &         0.8080 &   1.1045 \\ \hline
2  & 1.7522 & \textcolor{red}{\bf 1.0583} &   \textcolor{red}{\bf 1.1832} &         1.7621 &         1.7621 &  \textcolor{red}{\bf 1.1519} &         \textcolor{red}{\bf 1.1519} &         2.1761 &         1.7554 &   3.1241 \\ \hline
3  & 33.5742 & \textcolor{red}{\bf 33.4664} &        \textcolor{red}{\bf 33.4664} &        33.5742 &        33.5742 &        \textcolor{red}{\bf 33.4664} &         \textcolor{red}{\bf 5.9740} &        33.5849 &        33.5743 &  36.8782 \\ \hline
4  & 37.4166 & 37.4166 &        37.4166 &        37.4166 &        37.4166 &        37.4166 &        \textcolor{red}{\bf 26.4576} &        37.4166 &        37.4166 &  37.4166 \\ \hline
5  & 6.6330 & \textcolor{red}{\bf -2.7238} &        18.3602 &         7.5400 &         7.5400 &         9.2966 &        10.1397 &         6.7847 &         7.1170 &   7.5400 \\ \hline
6  & 88.8686 & \textcolor{red}{\bf -2.7238} &        \textcolor{red}{\bf 83.2762} &       716.2437 &       716.2437 &        \textcolor{red}{\bf 83.7136} &        89.9681 &       219.3651 &       262.0067 & 716.2437 \\ \hline
7  & 2.7483 & 2.7483 &   2.7483 &   2.7483 &   2.7483 &   2.7483 &   \textcolor{red}{\bf 1.6000}  &   2.7483 &   2.7483 &   2.7483 \\ \hline
\end{tabular}
\caption{Results for maximal wave speed $S_R$ . Column 2 displays the exact solution. Existing estimates are shown in columns
3 to 8, while colums 9 to 11 show the new estimates proposed in the present paper. A value in red indicates that the speed estimate
fails to be a bound for the exact solution.}
\label{Tab:Euler_Tests_SR}            
\end{center}
\end{table}
\begin{table}[h!]
\scriptsize
\begin{center}
\begin{tabular}{|c|c|c|c|c|c|c|c|c|c|c|} \hline
Test & $S^{Ex}_{L}$  & $S^{Dav_a}_{L}$ & $S^{Dav_b}_{L}$   & $S^{To}_{L}$ & $S^{GP}_{L}$ & $S^{Batten}_L$ &  $S^{Einf}_{L}$   & ${S}^{TMS_a}_{L}$  & ${S}^{TMS_b}_{L}$ & ${S}^{TMS_c}_{L}$ \\ \hline
1  & -1.1832  &  -1.1832 &        -1.1832 &        -1.1832 &        -1.1832 &        -1.1832 &        \textcolor{red}{\bf -0.8775} &        -1.1832 &        -1.1832 &  -1.1832 \\ \hline
2  & -1.1832 & -1.1832 &        -1.1832 &        -1.1832 &        -1.1832 &        -1.1832 &        \textcolor{red}{\bf -1.1519} &        -1.1832 &        -1.1832 &  -1.1832 \\ \hline
3  & -1.1832 & -1.1832 &       -33.4664 &        -1.1832 &        -1.1832 &        -5.9740 &        -5.9740 &        -1.1832 &        -1.1832 &  -1.1832 \\  \hline
4  & - 23.5175 & \textcolor{red}{\bf -0.1183} &       -37.4166 &       -33.0899 &       -33.0899 &       -26.4576 &       -26.4576 &       -31.7392 &       -33.0886 & -34.6410 \\ \hline
5  & -5.1678  & \textcolor{red}{\bf -2.3602} &        -9.2762 &        -6.0404 &        -6.0404 &        -7.2966 &        -8.1397 &        -5.3135 &        -5.6329 &  -6.0404 \\ \hline
6  & 70.4335  & \textcolor{red}{\bf 76.7238} &        -9.2762 &         7.7651 &         7.7651 &        60.6501 &        54.3955 &        57.4298 &        53.1710 &   7.7651 \\ \hline
7  & -2.7483 & -2.7483 &  -2.7483 &  -2.7483 &  -2.7483 &  -2.7483 &  \textcolor{red}{\bf - 1.6000} &  -2.7483 &  -2.7483 &  -2.7483 \\ \hline
\end{tabular}
\caption{Results for maximal wave speed $S_L$ . Column 2 displays the exact solution. Existing estimates are shown in culumns
3 to 8, while colums 9 to 11 show the new estimates proposed in the present paper. A value in red indicates that the speed estimate
fails to be a bound for the exact solution.}
\label{Tab:Euler_Tests_SL}            
\end{center}
\end{table}

Three new approaches to estimate wave speed bounds for the Euler equations have been proposed. These have been designated as $TMS_a$, $TMS_b$ and $TMS_c$. In the next two sections we apply these approaches to the shallow water equations and to the blood flow equations. For such systems we simplify further approach $TMS_c$  and introduce an even simpler approach called $TMS_d$. Aproaches $TMS_c$ and $TMS_d$ are indeed very simple, but not necessarily accurate, as we shall see in sections \ref{sec:boundsshallowwater} and \ref{sec:boundsbloodflow}.

\section{Speed Bounds: Shallow Water Equations}
\label{sec:boundsshallowwater}

In this section we propose bounds for minimal and maximal wave speeds emerging from the solution of the Riemann problem for for the shallow water equations. For background see \cite{Toro:2001a}. 

\subsection{Equations and wave relations}

Here we review the one-dimensional shallow water equations (SWEs) for the case of a horizontal bottom and state the relations that are used in the next sections to obtain wave speed bounds. Full details on the model adopted can be found in \cite{Toro:2001a} and references therein. The system of governing equations
is
\begin{equation}                                        \label{swe:cons}
	     \partial_{t}{\bf Q} +  \partial_{x} {\bf F}({\bf Q})={\bf 0} \;, 
\end{equation} 
with the vector of conserved variables ${\bf Q}(x,t)=\left[h, h u\right]^{T}$  and the flux vector ${\bf F(Q)}= \left[h u, h u^2 + \frac{1}{2}g h^{2}\right]^{T}$. Moreover, $x \in \mathbb{R}$ and $t \in [0,+\infty)$ are the space and time coordinates, $u\in\mathbb{R}$ is the fluid velocity, $h \in (0,+\infty)$ is the flow depth and $g$ is a parameter of the problem and represents the acceleration due to gravity, taken here as $g=9.8 m/s^{2}$.
The eigenvalues of the Jacobian associated to flux vector $\bf F(Q)$ in (\ref{swe:cons}) are all real and given by 
\begin{equation}                                        \label{swe3}
      \lambda_{1}=u-c \;, \hspace{3mm}  \;,  \hspace{3mm} \lambda_{2}=u+c \;,  
\end{equation}
where $c=\sqrt{gh}$ is the {\it celerity}. 

The Riemann problem is formulated as in (\ref{ee6}). For the SWEs considered here there are only two wave families emerging from the initial discontinuity. The corresponding two waves separate three constant states, namely ${\bf Q}_{L}$, ${\bf Q}_{*}$ and ${\bf Q}_{R}$.  The state ${\bf Q}_{*}$ is unknown.

Next we present four lemmas that will be used for proving that our proposed wave speed estimates are bounds for the minimal and maximal wave speeds.

\begin{lemma} \label{sweWaveRelations} 
\noindent{\bf Exact Riemann solution and wave speeds}  

\begin{enumerate}

\item {\bf Wave jumps across rarefactions: } across the left rarefaction  the left Riemann invariant gives
\begin{equation}                         \label{swe:10}
      u_{*} + 2c_{*} = u_{L} + 2c_{L}  \;,
\end{equation}
and across the right rarefaction we have
\begin{equation}                         \label{swe:11}
      u_{*} - 2c_{*} = u_{R} - 2c_{R}  \;.
\end{equation}
\item {\bf Wave jumps across shocks: } for a left shock, the Rankine-Hugoniot conditions give
\begin{equation}                          \label{swe:12} 
		u_{*} = u_{L} - f_{L}  \;; \hspace{3mm}
		f_{L}  = (h_{*}-h_{L})\sqrt{\frac{1}{2} g \left(\frac{h_{*}+h_{L}}{h_{*}h_{L}}\right)} \;,
\end{equation}
while for  a right shock one has

\begin{equation}                          \label{swe:14} 
		u_{*} = u_{R} + f_{R}  \;; \hspace{3mm}
		f_{R}  = (h_{*}-h_{R})\sqrt{\frac{1}{2} g \left(\frac{h_{*}+h_{R}}{h_{*}h_{R}}\right)} \;.
\end{equation}

\item {\bf Shock speeds: } the speed for a left shock is given as
\begin{equation}                          \label{swe:13}
S_{L} = u_{L} - c_{L} q_{L} \;, \hspace{3mm} 
       q_{L} = \sqrt{\frac{1}{2}\left(y^2+y\right)} \;,  \hspace{3mm} y= \frac{h_{*}}{h_{L}} \;.
\end{equation}
For a right shock the speed is
\begin{equation}                           \label{swe:15}
       S_{R} = u_{R} + c_{R} q_{R} \;, \hspace{3mm} 
       q_{R} = \sqrt{\frac{1}{2}\left(y^2+y\right)} \;,  \hspace{3mm} y= \frac{h_{*}}{h_{R}} \;.
\end{equation}
\item {{\bf Solution for flow depth} $h_*$: } solving the Riemann problem for system (\ref{swe:cons}) requires solving a nonlinear algebraic equation $f(h)=0$, with root $h_*$, with the function $f$ given as
\begin{equation}   \label{swe:starprob}
    f(h;h_L,h_R) = f_L(h;h_L) + f_R(h; h_R) + u_R - u_L\,,
\end{equation}
where $(h_L,u_L)$ and $(h_R,u_R)$ are the left/right states for the Riemann problem and
\begin{equation}  \label{swe:branches}
    f_K(h;h_K) = \left\{ \begin{array}{lllll} 
         f^{RAR}_K &= 2 (c(h)-c(h_K))\;, & \mbox{if } & h \le h_K    &  (rarefaction)\,,\\
         f^{SHO}_K &= \sqrt{\frac{1}{2} g \frac{(h + h_K)(h-h_{K})^2}{h h_K}}\;, & \mbox{if } & h > h_K  &  (shock) \;, 
    \end{array}\right.
\end{equation}
for $K=L$ or $K=R$.

\item {{\bf Two-rarefaction solution for flow depth} $h_*rr$.}  When both  branches $f_L$ and $f_R$ in (\ref{swe:branches}) are those of rarefaction waves one has
\begin{equation} \label{swe:twoRarFunc}
    f^{RAR}(h;h_L,h_R) = f^{RAR}_L(h;h_L) + f^{RAR}_R(h; h_R) + u_R - u_L\,,
\end{equation}
for which the closed-form solution of $f^{RAR}(h;h_L,h_R)=0$ is
\begin{equation}\label{swe:atr}
 h_{*rr} = \frac{1}{g}\left[\frac{1}{2}(c_L+c_R)+\frac{1}{4}(u_L-u_R)\right]^2\;.
\end{equation}

\end{enumerate}
\end{lemma}
\noindent{\bf Remark:} Obvioulsy  for the case in which both waves are rarefactions (\ref{swe:atr}) is the exact solution of (\ref{swe:starprob}), while for other cases it becomes an approximation, and hence the common name of {\it two-rarefaction approximation}.
\begin{proof}
Omitted. Full details are found in \citep{Toro:2001a}.
\end{proof}
\begin{lemma} \label{swe:theorar}
 $f(h;h_L,h_R) \ge f^{RAR}(h;h_L,h_R) \quad \forall h>0\,.$
\end{lemma}
\begin{proof}
The approach followed here is similar to the one proposed in \cite{Guermond:2015a}.
 First note that it is sufficient to prove that $f^{SHO}_K \geq f^{RAR}_K$. In fact, fixing $f_K=f^{RAR}_K$ with $f^{SHO}_K \geq f^{RAR}_K$ for either of the two $f_K$ in (\ref{swe:starprob}) fulfills the statement to be proved.
 We thus concentrate in proving
 \begin{equation}
     \left( \frac{1}{2} g \frac{(h+h_K)(h-h_K)^2}{h h_K} \right)^{\frac{1}{2}} \geq 2 (c(h)-c(h_K)) \,. \label{eq:toprove_SW}
 \end{equation}
 Before proceeding we introduce $y=h/h_K$ and note that 
 \begin{equation}
     2 (c(h)-c(h_K)) = 2 \left( g h_K \right)^{\frac{1}{2}}   (y^{\frac{1}{2}} -1)\,. \label{eq:rarSimp_SW}
 \end{equation}
 Moreover,
 \begin{equation} \label{eq:shockSimp_SW}
      \left( \frac{1}{2}  g \frac{(h+h_K)(h-h_K)^2}{h h_K} \right)^{\frac{1}{2}} = \left( \frac{1}{2}  g  \frac{h_K(y+1) h^2_K(y-1)^2}{h h_K} \right)^{\frac{1}{2}}=
     \left( g h_K \right)^{\frac{1}{2}} \left( \frac{1}{2} \frac{(y+1) (y-1)^2}{y} \right)^{\frac{1}{2}}\,.
 \end{equation}
 By replacing (\ref{eq:rarSimp_SW}) and (\ref{eq:shockSimp_SW}) in (\ref{eq:toprove_SW}) we obtain
 \begin{equation}  \label{swe:inequality}
       \left( \frac{1}{2} \frac{(y+1) (y-1)^2}{y} \right)^{\frac{1}{2}} \geq 2    (y^{\frac{1}{2}} -1)\,
 \end{equation}{}
which proves the claim for $y\leq 1$, since the left hand side is always positive while the right hand side is  negative. \\

Next, for $y>1$, we square both sides of inequality (\ref{swe:inequality}), that is
 \begin{equation}
       \frac{1}{2} \left( \frac{(y+1) (y-1)^2}{y} \right) \geq 4    (y^{\frac{1}{2}} -1)^2\,.
 \end{equation}{}
or
 \begin{equation}
          (y+1) (y-1)^2  \geq 8 y    (y^{\frac{1}{2}} -1)^2\, \label{eq:toprove2_SW} \;.
 \end{equation}
  We now rearrange the above expression as
 \begin{equation}
     (y+1) (y-1)^2  - 8 y    (y^{\frac{1}{2}} -1)^2 \geq 0\,,
 \end{equation}{}
 and factor it as
\begin{equation} \label{swe:ineq}
     (\sqrt{y} - 1)^4 (1 + 4 \sqrt{y} + y) \geq 0\,,    
 \end{equation}{}
which holds true for $y>0$ since all terms are greater or equal to zero for $y>0$. This concludes the proof.
\end{proof}

\begin{lemma} \label{swe:theoshockspeed}
The function $q_K$ in (\ref{swe:13}) and (\ref{swe:15}), with $K=\{L,R\}$, is monotone increasing in $h\;,\forall h>0\,.$
\end{lemma}
\begin{proof}
Since $y>0$ by definition, the proof follows immediately.
 \end{proof}
 
\begin{lemma} \label{swe:theoConcaveDown}
 The function (\ref{swe:starprob}) is concave down $\forall h>0\,.$
\end{lemma}
\begin{proof}
Omitted. Full details can be found in \citep{Toro:2001a}.
\end{proof}

In the rest of this section we propose estimates for speed bounds of the shallow water system. We propose four new estimates for wave speed bounds, which we call: $TMS_a$, $TMS_b$, $TMS_c$ and $TMS_d$. We now proceed to formulate each one of these estimates and to prove our statements. Before proceeding with the definition of wave speed estimates, we first define some quantities that will be used repeatedly, namely
\begin{equation}
 h_{min}=\min\{h_L,h_R\}\;, \hspace{4mm} h_{max} = \max\{h_L,h_R\}\;
\end{equation}
and 
\begin{equation}
 f_{min}=f(h_{min})\;, \hspace{4mm} f_{max} = f(h_{max})\;.
\end{equation}
We also define the following three points:
\begin{equation}                                                 \label{sweTMSc4} 
    P_{m}=(h_{min}, f(h_{min})) \;, \hspace{3mm} P_{M}=(h_{max}, f(h_{max})) \;,  \hspace{3mm} P_{rr}=(h_{*rr}, f(h_{*rr})) \;.
\end{equation}
As for the Euler equations, these points will be selectively used for linear interpolation to obtain an approximation of the flow depth $h_{*}$.

\subsection{SWE speed bound estimate: approach $TMS_a$}

This estimate requires up to three evaluations of function (\ref{swe:starprob}). Table \ref{tab:sweTMSa} summarizes the results presented here.

\begin{theorem} \label{swe:theoTMSb}
 \noindent{\bf Method $TMS_a$ for SWE:} The exact left and right 
wave speeds $S_{L}^{Ex}$ and $S_{R}^{Ex}$ are bounded by $S_{L}^{b}$ and $S_{R}^{b}$, computed as follows:

\begin{itemize}
\item{\bf Case R/R: two rarefaction waves.} In this case we have that  $f_{min}\geq 0$. Then we set
 \begin{equation} \label{swe:tsmcRR}
  S_{L}^{b} = S^{Ex}_L = u_L-c_L\;, \hspace{5mm} S_{R}^{b} = S_{R}^{Ex} = u_R+c_R\;.
 \end{equation}
 \item{\bf Case R/S: left rarefaction and right shock.} The conditions are:  $f_{min}<0$, $f_{max}>0$ and $h_{max}=h_L$. Then one computes
     \begin{equation}\label{swe:tsmcSRLin}
        h_{*mM} = h_{min} - \frac{h_{max}-h_{min}}{f_{max}-f_{min}} f_{min}\;
    \end{equation}
    and then set
     \begin{equation}\label{swe:tsmcRS}
        S_{L}^{b} = u_L- c_L \;, \hspace{5mm} S_{R}^{b} = u_R + c_R q_{R}(h_{*mM})\;.
     \end{equation}
\item{\bf Case S/R: left shock and right rarefaction.} The conditions are: $f_{min}<0$, $f_{max}>0$ and $h_{max}=h_L$. Then we set
     \begin{equation}\label{swe:tsmcRS}
        S_{L}^{b} = u_L- c_Lq_{L}(h_{*mM})\;, \hspace{5mm} S_{R}^{b} = u_R + c_R \;.
     \end{equation}
\item{\bf Case S/S: two shock waves.} If $f_{max}\leq0$ one first computes
 \begin{equation}\label{swe:tsmcSSLin}
        h_{*Mrr} = h_{max} -  \frac{h_{*rr}-h_{max}}{f(h_{*rr})-f_{max}} f_{max}\;
    \end{equation}
 and then sets
 \begin{equation}\label{swe:tsmcSS}
  S_{L}^{b} = u_L - c_L q_{L}(h_{*Mrr})\;, \hspace{5mm} S_{R}^{b} = u_R + c_R q_{R}(h_{*Mrr})\;.
 \end{equation}
\end{itemize}
\end{theorem}

\begin{proof}

We consider the four possible wave configurations separately.
\paragraph{$TMS_a$ - Proof for case of left rarefaction/right rarefaction (R/R)}
 
Wave speed estimates (\ref{swe:tsmcRR}) are identical to the exact wave speeds for this wave configuration. 

\paragraph{$TMS_a$ - Proof for case of left rarefaction/right shock (R/S)}
 
Wave speed estimates in (\ref{swe:tsmcRS}) for $S^b_L$ is identical to the exact wave speed for this wave type.
Wave speed estimate in (\ref{swe:tsmcRS}) for $S^b_R$ is computed with the same expression used for the exact wave speed for this wave type, see (\ref{swe:15}) , but using $h_{*mM}$ instead of $h_*$. The resulting estimate is a bound since for this wave configuration we have that:  (\ref{swe:15}) is monotone increasing in $h$ (from Lemma \ref{swe:theoshockspeed}) and $f(h;h_L,h_R)$ is concave down (from Lemma \ref{swe:theoConcaveDown}).
 
\paragraph{$TMS_a$ - Proof for case of left shock/right rarefaction (S/R)}
 
Proving that expressions in (\ref{swe:tsmcRS}) are bounds is entire analogous to the previous case. 
 
\paragraph{$TMS_a$ - Proof for case of left shock/right shock (S/S)}
 
Wave speed estimates in (\ref{swe:tsmcRS}) are computed with the same expression used for the exact wave speeds for this wave type, see (\ref{swe:15}) , but using $h_{*mrr}$ instead of $h_*$. The resulting estimates are bounds since for this wave configuration we have that:  $h_{*rr}>h_*$ (from Lemma \ref{swe:theorar}), (\ref{swe:15}) is monotone increasing in $h$ (from Lemma \ref{swe:theoshockspeed}) and $f(h;h_L,h_R)$ is concave down (from Lemma \ref{swe:theoConcaveDown}).
 \end{proof}
\begin{table}
\begin{center}
\begin{tabular}{|c|c|c|c|} \hline
Wave pattern   &   Conditions       &    $S_{L}^{b}$     &   $S_{R}^{b}$  \\ \hline
$R/R$   &  $f(h_{min})\ge 0$     &    $u_{L} - c_{L}$                    &   $u_{R} + c_{R}  $  \\ \hline 
$R/S$   &  $f(h_{min}) < 0$ \;, $f(h_{max}) > 0$ \;, $h_{min}=h_{R}$ &  $u_{L} -  c_{L}$  &  $ u_{R} +  c_{R}q_{R}(h_{*mM}) $ \\ \hline
$S/R$ &  $f(h_{min}) < 0$ \;, $f(h_{max}) > 0$ \;, $h_{min}=h_{L}$    &  $ u_{L} -  c_{L}q_{L}(h_{*mM}) $  & $  u_{R} +  c_{R} $ \\ \hline
$S/S$  &  $f(h_{max}) < 0$  &   $  u_{L} -  c_{L}q_{L}(h_{*Mrr})   $           &$   u_{R} + c_{R}q_{R}(h_{*Mrr})  $  \\ \hline 
\end{tabular}       
\caption{$TMS_a$ bound estimates $S_{L}^{b}$ and $S_{R}^{b}$ on minimal and maximal wave speeds for the blood flow equations (\ref{bfe:cons}). Function $q_{K}(h)$ ($K=L, R$) given in Eqs. (\ref{swe:13})-(\ref{swe:15}). Value $h_{*mM}$ is given in Eq. (\ref{swe:tsmcSRLin}). Value  $h_{*Mrr}$ is given in Eq. (\ref{swe:tsmcSSLin}).}\label{tab:sweTMSa}
\end{center}
\end{table}

\subsection{SWE speed bound estimate: approach $TMS_b$}

This estimate is similar to $TMS_a$. It requires up to three evaluations of function (\ref{swe:starprob}). Table \ref{tab:sweTMSb} summarizes the results presented here.

\begin{theorem} \label{swe:theoTMSc}
 
\noindent{\bf Method $TMS_b$ for SWE:} The exact left and right 
wave speeds $S_{L}^{Ex}$ and $S_{R}^{Ex}$ are bounded by $S_{L}^{b}$ and $S_{R}^{b}$, computed as follows:
\begin{itemize}
    \item{\bf Case R/R: two rarefaction waves.} In this case we have that  $f_{min}\geq0$. Then we set
    \begin{equation} \label{swe:tsmdRR}
    S_{L}^{b} = S^{Ex}_L = u_L-c_L\;, \hspace{5mm} S_{R}^{b} = S_{R}^{Ex} = u_R+c_R\;.
    \end{equation}

\item{\bf Case R/S: left rarefaction and right shock.} $f_{min}<0$, $f_{max}>0$ and $A_{max}=A_L$. Then we compute
    \begin{equation}\label{swe:tsmdLin}
        h_{*mrr} = h_{min} -  \frac{h_{*rr}-h_{min}}{f(h_{*rr};h_L,h_R)-f_{min}} f_{min}\;.
    \end{equation}
    Then we set
    \begin{equation}\label{swe:tsmdRS}
        S_{L}^{b} = u_L- c_L\;, \hspace{5mm} S_{R}^{b} = u_R + c_R q_{R}(h_{*mrr})\;.
    \end{equation}
    \item{\bf Case S/R: left shock and right rarefaction.} $f_{min}<0$, $f_{max}>0$ and $h_{max}=h_R$. Then we set
    
    \begin{equation}\label{swe:tsmdSR}
    S_{L}^{b} = u_L - c_L q_{L}(h_{*mrr}) \;, \hspace{5mm} S_{R}^{b} = u_R + c_R\;.
    \end{equation}
\item{\bf Case S/S: two shock waves.} If $f_{max}\leq0$ we set
 \begin{equation}\label{swe:tsmcSS}
  S_{L}^{b} = u_L - c_L q_{L}(h_{*mrr})\;, \hspace{5mm} S_{R}^{b} = u_R + c_R q_{R}(h_{*mrr})\;.
 \end{equation}
\end{itemize}
\end{theorem}

\begin{proof}
We consider the four possible wave configurations separately.

\paragraph{$TMS_b$ - Proof for case of left rarefaction/right rarefaction (R/R)}
 Wave speed estimates (\ref{swe:tsmcRR}) are identical to the exact wave speeds for this wave configuration. 

\paragraph{$TMS_b$ - Proof for case of left rarefaction/right shock (R/S)}
 
Wave speed estimates in (\ref{swe:tsmcRS}) for $S^b_L$ is identical to the exact wave speed for this wave type.
Wave speed estimate in (\ref{swe:tsmcRS}) for $S^b_R$ is computed with the same expression used for the exact wave speed for this wave type, see (\ref{swe:15}) , but using $h_{*mrr}$ instead of $h_*$. The resulting estimate is a bound since for this wave configuration we have that:  (\ref{swe:15}) is monotone increasing in $h$ (from Lemma \ref{swe:theoshockspeed}) and $f(h;h_L,h_R)$ is concave down (from Lemma \ref{swe:theoConcaveDown}).
 
\paragraph{$TMS_b$ - Proof for case of left shock/right rarefaction (S/R)} Omitted, see previous case.

\paragraph{$TMS_b$ - Proof for case of left shock/right shock (S/S)}
 Wave speed estimates in (\ref{swe:tsmcRS}) are computed with the same expression used for the exact wave speeds for this wave type, see (\ref{swe:15}) , but using $h_{*mrr}$ instead of $h_*$. The resulting estimates are bounds since for this wave configuration we have that:  $h_{*rr}>h_*$ (from Lemma \ref{swe:theorar}), (\ref{swe:15}) is monotone increasing in $h$ (from Lemma \ref{swe:theoshockspeed}) and $f(h;h_L,h_R)$ is concave down (from Lemma \ref{swe:theoConcaveDown}).
 
\end{proof}
\begin{table}
\begin{center}
\begin{tabular}{|c|c|c|c|} \hline
Wave pattern   &   Conditions       &    $S_{L}^{b}$     &   $S_{R}^{b}$  \\ \hline
$R/R$          &  $f(h_{min})\ge 0$ &    $u_{L} - c_{L}$ &   $u_{R} + c_{R}  $  \\ \hline 
$R/S$          &  $f(h_{min}) < 0$ \;, $f(h_{max}) > 0$ \;, $h_{min}=h_{R}$ &  $u_{L} -  c_{L}$  &  $ u_{R} +  c_{R}q_{R}(h_{*mrr}) $ \\ \hline
$S/R$ &  $f(h_{min}) < 0$ \;, $f(h_{max}) > 0$ \;, $h_{min}=h_{L}$    &  $ u_{L} -  c_{L}q_{L}(h_{*mrr}) $  & $  u_{R} +  c_{R} $ \\ \hline
$S/S$  &  $f(h_{max}) < 0$  &   $  u_{L} -  c_{L}q_{L}(h_{*mrr})   $           &$   u_{R} + c_{R}q_{R}(h_{*mrr})  $  \\ \hline 
\end{tabular}       
\caption{$TMS_b$ bound estimates $S_{L}^{b}$ and $S_{R}^{b}$ on minimal and maximal wave speeds for the shallow water equations (\ref{swe:cons}). Function $q_{K}( h)$ ($K=L, R$), given in Eqs. (\ref{swe:13})-(\ref{swe:15}). Function  $h_{*mrr}$ is given in Eq. (\ref{swe:tsmdLin})}.
\label{tab:sweTMSb}
\end{center}
\end{table}

\subsection{SWE speed bound estimate: approach $TMS_c$} 

This method is a simplification of $TMS_c$ presented for the Euler equations. Table \ref{tab:sweTMSc} summarizes the results presented here. 

\begin{theorem} \label{swe:theoTMSc}
 
\noindent{\bf Method $TMS_c$ for SWE:} The exact left and right 
wave speeds $S_{L}^{Ex}$ and $S_{R}^{Ex}$ are bounded by $S_{L}^{b}$ and $S_{R}^{b}$, computed as follows:
\begin{itemize}
 \item{\bf Case R/R: two rarefaction waves.} In this case we have that  $f_{min}\geq0$. Then one sets
 \begin{equation} \label{swe:tsmbRR}
  S_{L}^{b} = S^{Ex}_L = u_L-c_L\;, \hspace{5mm} S_{R}^{b} = S_{R}^{Ex} = u_R+c_R\;.
 \end{equation}
 \item{\bf Case R/S: left rarefaction and right shock.} The conditions are: $f_{min}<0$, $f_{max}>0$ and $h_{max}=h_L$. Then one sets
     \begin{equation}\label{swe:tsmbRS}
        S_{L}^{b} = u_L- c_L\;, \hspace{5mm} S_{R}^{b} = u_R + c_R q_{R}(h_L)\;.
     \end{equation}
 \item{\bf Case S/R: left shock and right rarefaction.} The conditions are: $f_{min}<0$, $f_{max}>0$ and $h_{max}=h_R$. Then one sets 
     
     \begin{equation}\label{swe:tsmbSR}
       S_{L}^{b} = u_L - c_L q_{L}(h_R) \;, \hspace{5mm} S_{R}^{b} = u_R + c_R\;.
     \end{equation}
\item{\bf Case S/S: two shock waves.} If $f_{max}\leq0$ one sets
 \begin{equation}\label{swe:tsmbSS}
  S_{L}^{b} = u_R- c_R\;, \hspace{5mm} S_{R}^{b} = u_L + c_L\;.
 \end{equation}
 
\end{itemize}
\end{theorem}
\begin{proof}

We consider the four possible wave configurations separately.

\paragraph{$TMS_c$ - Proof for case of left rarefaction/right rarefaction (R/R)}
 
Wave speed estimates (\ref{swe:tsmbRR}) are identical to the exact wave speeds for this wave configuration.

\paragraph{$TMS_c$ - Proof for case of left rarefaction/right shock (R/S)}
 
Wave speed estimates in (\ref{swe:tsmbRS}) for $S^b_L$ is identical to the exact wave speed for this wave type.
Wave speed estimate in (\ref{swe:tsmbRS}) for $S^b_R$ is computed with the same expression used for the exact wave speed for this wave type, see (\ref{swe:15}) , but using $h_L$ instead of $h_*$. The resulting estimate is a bound since for this wave configuration $h_L>h_*$ and (\ref{swe:15}) is monotone increasing in $h$ (see Lemma \ref{swe:theoshockspeed}).
 
\paragraph{$TMS_c$ - Proof for case of left shock/right rarefaction (S/R)}
 
Proving that expressions in (\ref{swe:tsmbSR}) are bounds is entirely analogous to the previous case. 

\paragraph{$TMS_c$ - Proof for case of left shock/right shock (S/S)}
 
The Lax entropy condition ensures 
\begin{equation}                                                 \label{swe:28} 
		u_{*} + c_{*} > S_{R}^{Ex}\;.
\end{equation}
The task is to find a bound for the characteristic speed $u_{*} + c_{*} $ in the start region.  Recalling our hypothesised bound, we want to prove that
\begin{equation}                                                 \label{swe:29} 
		S_{R}^{b} = u_{L} + c_{L}  \ge u_{*} + c_{*}   \;.
\end{equation}
We recall that for the left shock wave (\ref{swe:13}) holds, such relation reads
\begin{equation}                                                 \label{swe:30}
		u_{*} = u_{L} - f_L = u_{L} -  (h_{*}-h_{L}) \sqrt{\frac{1}{2} g \left(\frac{h_{*}+h_{L}}{h_{*}h_{L}}\right)} = u_{L} -  (y-1) \sqrt{\frac{1}{2} g h_L \left(1+y\right)} ,
		\hspace{5mm}  y \equiv \frac{h_{*}}{h_{L}} \;.
\end{equation}
Substitution of  $u_{*}$  into (\ref{swe:29}), division  by $c_{L}$ and simple manipulations give
\begin{equation}                                                 \label{swe:31a} 
		1 \ge \frac{c_{*}}{c_{L}}-\frac{f_{L}}{c_{L}} \;,
\end{equation}
which can be rewritten as
\begin{equation}                                                 \label{swe:31_b} 
		\left(\sqrt{y} -1\right)  \le \left(y-1\right) \sqrt{\frac{1}{2}(1+y)} \;.
\end{equation}
Re-writing the right hand side we obtain:
\begin{equation}                                 \label{swe:31_c} 
		 \left(\sqrt{y} -1\right)  \le \left(\sqrt{y} -1\right) \left(\sqrt{y} +1\right) \sqrt{\frac{1}{2}(1+y)} \;
\end{equation}
and then
\begin{equation}                              \label{swe:31_d} 
		 1  \le \left(\sqrt{y} +1\right) \sqrt{\frac{1}{2}(1+y)} \;.
\end{equation}
Since we are considering a S/S condition we have that $y>1$ and therefore inequality (\ref{swe:31_d}) is valid. This concludes the prove that $ S^b_R$ in (\ref{swe:29}) is a bound. 

The proof for $S^b_L$ is analogous to the one performed for $S^b_R$ and is thus omitted. As all four possible configurations have been considered the proof is complete. Hence, we obtain that (\ref{swe:29}) results in a bound for $S_R^e$. With this last case we have covered all four possible configurations and the proof is complete.
\end{proof}
\begin{table}
\begin{center}
\begin{tabular}{|c|c|c|c|} \hline
Wave pattern   &   Conditions       &    $S_{L}^{b}$     &   $S_{R}^{b}$  \\ \hline
$R/R$   &  $f(h_{min})\ge 0$     &    $u_{L} - c_{L}$                    &   $u_{R} + c_{R}  $  \\ \hline 
$R/S$   &  $f(h_{min}) < 0$ \;, $f(h_{max}) > 0$ \;, $h_{min}=h_{R}$ &  $u_{L} -  c_{L}$  &  $ u_{R} +  c_{R}q_{R}(h_{L}) $ \\ \hline
$S/R$ &  $f(h_{min}) < 0$ \;, $f(h_{max}) > 0$ \;, $h_{min}=h_{L}$    &  $ u_{L} -  c_{L}q_{L}(h_{R}) $  & $  u_{R} +  c_{R} $ \\ \hline
$S/S$  &  $f(h_{max}) < 0$  &   $  u_{R} -  c_{R}   $           &$   u_{L} + c_{L}  $  \\ \hline 
\end{tabular}       
\caption{$TMS_c$ bound estimates $S_{L}^{b}$ and $S_{R}^{b}$ on minimal and maximal wave speeds for the blood flow equations (\ref{swe:cons}). Function $q_{K}(h)$ ($K=L, R$) given in Eqs. (\ref{swe:13})-(\ref{swe:15}). }\label{tab:sweTMSc}
\end{center}

\end{table}

\subsection{SWE speed bound estimate: approach $TMS_d$}

This method is new and was not applied to the Euler equations above. The estimate is given explicitly in terms of left and right state vectors and uses very little information about the Riemann problem.

\begin{theorem}  \label{swe:theoTMSa}
\noindent{\bf Method $TMS_d$ for SWE:} The exact left and right 
wave speeds $S_{L}^{Ex}$ and $S_{R}^{Ex}$ are bounded by $S_{L}^{b}$ and $S_{R}^{b}$  respectively, with
\begin{equation}                           \label{swe:16} 
		S_{L}^{b} = \min\{u_{L} - c_{L},u_{R} - \alpha_L c_{R}\} \le S_{L}^{e} \;, \hspace{3mm} \alpha_{L}=2
\end{equation}
and
\begin{equation}                            \label{swe:17} 
		S_{R}^{b} = \max\{u_{R} + c_{R}, u_{L} + \alpha_{R} c_{L} \} \ge S_{R}^{b}\;, \hspace{3mm} \alpha_{R}=2 \;.
\end{equation}
\end{theorem}

\begin{proof}

We consider the four possible wave configurations separately.

\paragraph{$TMS_d$ - Proof for case of left rarefaction/right rarefaction (R/R)}

The proof in this case is trivial. If the exact solution consists of two rarefaction waves, then the exact wave speeds will be
\begin{equation}
 S_{L}^{Ex} = u_L-c_L\;, \hspace{5mm} S_{R}^{Ex} = u_R+c_R\;.
\end{equation}

Clearly (\ref{swe:16}) and (\ref{swe:17}) are bounds to these wave speeds.

\paragraph{$TMS_d$ - Proof for case of left rarefaction/right shock (R/S)} 

Application of the Lax entropy condition gives
\begin{equation}                             \label{swe:18} 
		u_{*} + c_{*} > S_{R}^{Ex}\;.
\end{equation}
The task is to find a bound for the characteristic speed $u_{*} + c_{*} $ in the {\it Star Region}.  We seek a bound of the form
\begin{equation}                              \label{swe:19} 
		S_{R}^{b} = u_{L} + \alpha_{R} c_{L}  \ge u_{*} + c_{*}   \;,
\end{equation}
with $\alpha_{R}$ yet to be found. From the left Riemann invariant (\ref{swe:10}) we may write
\begin{equation}                               \label{swe:20} 
		u_{*} = u_{L} - 2 (c_{*}-c_{L}) \;
\end{equation}
and thus the characteristic speed becomes
\begin{equation}                                \label{swe:21} 
	 u_{*} +c_{*} = u_{L} +2c_{L} - c_{*} \;.
\end{equation}
By imposing (\ref{swe:19}) we have
\begin{equation}                                \label{swe:22} 
		u_{L} +2c_{L} - c_{*}  \le u_{L} + \alpha_{R} a_{L} \;.
\end{equation}
Simple manipulations lead to
\begin{equation}                                \label{swe:23} 
	 2 - \sqrt{y} \le \alpha_{R}\;, \hspace{3mm}  y=\frac{h_{*}}{h_{L}} \;.
\end{equation}
Since the left wave is a rarefaction $0 < y < 1$. The most restrictive case is $\alpha_R=2$ and the result follows (see \cite{Toro:2001a} for related discussion on dry (vacuum) fronts).

\paragraph{$TMS_d$ - Proof for case of left shock / right rarefaction (S/R)} 
The proof for this configuration is analogous to the previous one for the R/S configuration and it is thus omitted.

\paragraph{$TMS_d$ - Proof for case of left shock / right shock (S/S)} 
The proof is analogous to that for the same wave configuration in Theorem \ref{swe:theoTMSb} and is thus omitted. 

\end{proof}

\subsection{Numerical tests for the shallow water equations}

In this section we perform some numerical tests to assess the performance of existing and newly proposed wave speed estimates. We do so through five Riemann problems, the initial conditions of which are given in Table \ref{Tab:SWE_Tests_IC}. The tests have been chosen so that all possible wave patterns are included.
Numerical results for the maximal and minimal wave speeds are shown in Tables \ref{Tab:SWE_Tests_SR} and \ref{Tab:SWE_Tests_SL} respectively. The results confirm that all proposed estimates constitute bounds for the extreme wave speeds, while existing estimates (\ref{Davis:a}) and (\ref{Davis:b}) fail to do so in general.
Our estimates $TMS_a$ and $TMS_b$ perform very satisfactorily, especially in the presence of very strong shocks, as for Test 3. Our proposed simpler bounds  $TMS_c$ and $TMS_d$ are less accurate, with $TMS_c$ giving more accurate estimates than $TMS_d$ for almost all cases, except for Test 2, where the left shock speed is grossly overestimated by $TMS_c$. In general, $TMS_a$ and $TMS_b$ perform better than other estimates. Test 4 illustrates the improvement resulting from using a linear interpolation and not just the two-rarefaction solution when comparing $To$ and $GP$ estimates with $TMS_b$.
\begin{table}[h!]    
\begin{center}
\begin{tabular}{|c|c|c|c|c|c|c|c|} \hline
Test &  $h_{L}$  & $u_{L}$  & $h_{R}$  & $u_{R}$ & $h_*$  & $u_*$  & wave pattern  \\ \hline
1    & 1.0000    & 0.0000   &  0.7000  & 0.0000  & 0.8430 & 0.5121 & rar-shock     \\ \hline
2    & 0.0010    & 0.0000   &  1.0000  & 0.0000  & 0.0668 & -4.6424 & shock-rar     \\ \hline
3    & 1.0000    & 3.0000   &  0.5000  & 0.0000  & 1.1671 & 2.4959 & shock-shock    \\ \hline 
4    & 1.0000    & 100.0000   &  0.5000  & 0.0000  & 19.0839 & 58.9340 & shock-shock    \\ \hline
5    & 1.0000    & -5.0000  &  1.0000  & 5.0000  & 0.0406 & 0.0000 & rar-rar         \\ \hline
\end{tabular}     
\caption{Initial conditions for five Riemann problems. Columns $6$ and $7$ show the exact solutuion for depth $h_{*}$ and velocity $u_{*}$ in the {\it Star Region}. The type of emerging wave patter is shown in the last column. SI units used.}
\label{Tab:SWE_Tests_IC}     
\end{center}
\end{table}
\begin{table}[h!]
\begin{center}
\begin{tabular}{|c|c|c|c|c|c|c|c|c|c|} \hline
Test & $S^{Ex}_{R}$  & $S^{Dav_a}_{R}$ & $S^{Dav_b}_{R}$   & $S^{To}_{R}$ & $S^{GP}_{R}$   & ${S}^{TMS_a}_{R}$  & ${S}^{TMS_b}_{R}$ & ${S}^{TMS_c}_{R}$ & ${S}^{TMS_d}_{R}$\\ \hline
1  & 3.0177  & \textcolor{red}{\bf 2.6192} &   3.1305 &   3.0184 &   3.0184 &   3.0344 &   3.0178 &   3.4496 &   6.2610 \\ \hline
2  & 3.1305  & 3.1305           &  3.1305 &  3.1305 &  3.1305 &  3.1305 &  3.1305 &  3.1305 &  3.1305 \\ \hline
3  & 4.3667  & \textcolor{red}{\bf 2.2136} &   6.1305 &   4.4552 &   4.4552 &   4.3686 &   4.3789 &   6.1305 &   9.2610 \\ \hline
4  & 60.5197  & \textcolor{red}{\bf 2.2136} & 103.1305 & 245.3887 & 245.3887 &  61.1536 &  61.6620 & 103.1305 & 106.2610 \\ \hline
5  & 8.1305  & 8.1305           &  8.1305           &   8.1305 &   8.1305 &   8.1305 &   8.1305 &   8.1305 &   8.1305 \\ \hline
\end{tabular}
\caption{Results for maximal wave speed $S_R$. Column $2$ displays the exact solution. Existing estimates are shown in columns $3$ to $6$,
while colums $7$ to 10 show the new estimates proposed in this paper. A value in red indicates that the speed estimate fails to be a bound for the exact solution.
SI units used.}
\label{Tab:SWE_Tests_SR}            
\end{center}
\end{table}

\vspace{0.1cm}

\vspace{1mm}
\begin{table}[h!]
\begin{center}
\begin{tabular}{|c|c|c|c|c|c|c|c|c|c|} \hline
Test & $S^{Ex}_{L}$  & $S^{Dav_a}_{L}$ & $S^{Dav_b}_{L}$   & $S^{To}_{L}$ & $S^{GP}_{L}$   & ${S}^{TMS_a}_{L}$  & ${S}^{TMS_b}_{L}$ & ${S}^{TMS_c}_{L}$ & ${S}^{TMS_d}_{L}$\\ \hline
1  & -3.1305  & -3.1305           &  -3.1305 &  -3.1305 &  -3.1305 &  -3.1305 &  -3.1305 &  -3.1305 &  -5.2383 \\ \hline
2  & -4.7130  & \textcolor{red}{\bf - 0.0990} &   \textcolor{red}{\bf - 3.1305} &  -18.6593 &  -18.6593 &   -5.6816 &   -5.3080 &  -70.0350 &   -6.2610  \\ \hline
3  & -0.5204  & \textcolor{red}{\bf -0.1305} &  -2.2136 &  -0.5849 &  -0.5849 &  -0.5218 &  -0.5293 &  -2.2136 &  -4.4272 \\ \hline
4  & 56.6632 & \textcolor{red}{\bf 96.8695} &  -2.2136 & -74.0668 & -74.0668 &  56.2149 &  55.8554 &  -2.2136 &  -4.4272 \\ \hline
5  & -8.1305  & -8.1305           &  -8.1305 &  -8.1305 &  -8.1305 &  -8.1305 &  -8.1305 &  -8.1305 &  -8.1305 \\ \hline
\end{tabular}
\caption{Results for minimal wave speed $S_L$. Column $2$ displays the exact solution. Existing estimates are shown in columns $3$ to $6$,
while colums $7$ to 10 show the new estimates proposed in this paper. A value in red indicates that the speed estimate fails to be a bound for the exact solution.
SI units used.}
\label{Tab:SWE_Tests_SL}            
\end{center}
\end{table}

\section{Speed Bounds: Blood Flow Equations}
\label{sec:boundsbloodflow}

\subsection{Equations and wave relations}

We first briefly review a widely used one-dimensional blood flow model for arteries, hereafter called the Blood Flow Equations (BFEs) and state the main relations that we later use to obtain wave speed bounds. For full details on the equations see for example \cite{Toro:2012a}, \cite{Toro:2013a}, \cite{Toro:2016a} and references therein. We consider a system of conservation laws 
\begin{equation}                            \label{bfe:cons}
      \partial_{t} {\bf Q}(x,t) + \partial_{x} {\bf F} ({\bf Q}(x,t)) = {\bf 0}  \;
\end{equation}
representing mass conservation and momentum balance in a vessel with walls exhibiting an elastic behaviour. The vector of conserved variables is ${\bf Q}(x,t)=\left[A, A u\right]^{T}$  and the flux vector is ${\bf F(Q)}= \left[A u, Au^2+\gamma A^{\frac{3}{2}}\right]^{T}$. Moreover, $x \in \mathbb{R}$ and $t \in [0,+\infty)$ are the space and time coordinates, $u\in\mathbb{R}$ is the fluid velocity and $A \in (0,+\infty)$ is the cross-sectional area of the vessel. Here
$\gamma$ is a constant parameter given by
\begin{equation}
    \gamma = \frac{\beta}{3\rho}\,,
\end{equation}
with $\beta$ being a positive constant and $\rho$ the fluid density, so $\gamma \in (0,+\infty)$.
We also define the wave speed $c$ (analogous to the sound speed in gas dynamics, or the celerity in shallow water) as 
\begin{equation}
    c = c(A) =\sqrt{\frac{3}{2} \gamma A^{\frac{1}{2}}} = \sqrt{\frac{3}{2} \frac{\beta}{3\rho} A^{\frac{1}{2}}} = \sqrt{\frac{\beta}{2\rho}} A^{\frac{1}{4}} = \zeta A^{\frac{1}{4}}
    \,, \hspace{3mm} \mbox{with } \zeta=\sqrt{\frac{\beta}{2\rho}}\;.
\end{equation}

The eigenvalues of the Jacobian matrix corresponding to the flux vector $\bf F(Q)$ in (\ref{bfe:cons}) are all real and given by 
\begin{equation}                       \label{bfe:eigs}
    \lambda_{1} = u - c \;,  \hspace{3mm} \lambda_{2} = u + c \;.
\end{equation}

The Riemann problem for (\ref{bfe:cons}) is formulated as in (\ref{ee6}). However, for the BFEs
there are only two wave families emerging from the initial discontinuity, separating three constant states, namely ${\bf Q}_{L}$, ${\bf Q}_{*}$ and ${\bf Q}_{R}$. The unknown state is ${\bf Q}_{*}$.

Next we present four lemmas that will be used for proving that some proposed wave speed estimates are bounds. 

 \begin{lemma} \label{bfeWaveRelations} 
\noindent{\bf Exact Riemann solution and wave speeds.} 

\begin{enumerate}

\item {\bf Wave jumps across rarefactions:} across the left rarefaction  the left Riemann invariant gives
\begin{equation}                         \label{bfe:10}
     u_{*} + 4c_{*} = u_{L} + 4c_{L}  \;
\end{equation}
and across the right rarefaction the right Riemann invariant gives
\begin{equation}                         \label{bfe:11}
     u_{*} - 4c_{*} = u_{R} - 4c_{R}  \;.
\end{equation}

\item  {\bf Wave jumps across shocks:} 
for a left shock, the Rankine-Hugoniot conditions give
\begin{equation}                          \label{bfe:12} 
		u_{*} = u_{L} - f_{L}  \;; \hspace{3mm}
		f_{L}  =  \sqrt{\frac{\gamma (A_{*}-A_{L})(A_{*}^{3/2}-A_{L}^{3/2})}{A_{L}A_{*}}} \;,
\end{equation}
while for a right shock one has 
\begin{equation}                          \label{bfe:14} 
		u_{*} = u_{R} + f_{R}  \;; \hspace{3mm}
		f_{R}  =  \sqrt{\frac{\gamma (A_{*}-A_{R})(A_{*}^{3/2}-A_{R}^{3/2})}{A_{R}A_{*}}} \;.
\end{equation}
\item {\bf Shock speeds:}
the shock speed for a left shock is given as
\begin{equation}                          \label{bfe:13}
S_{L} = u_{L} - c_{L} q_{L} \;, \hspace{3mm} 
      q_{L} = \sqrt{\frac{2}{3}\frac{(y^{3/2}-1)y}{(y-1)}} \;,  \hspace{3mm} y= \frac{A_{*}}{A_{L}} \;.
\end{equation}
The speed of a right shock is
\begin{equation}                           \label{bfe:15}
      S_{R} = u_{R} + c_{R} q_{R} \;, \hspace{3mm} 
      q_{R} = \sqrt{\frac{2}{3}\frac{(y^{3/2}-1)y}{(y-1)}} \;,  \hspace{3mm} y= \frac{A_{*}}{A_{R}} \;.
\end{equation}
\item {\bf Solution for area $A_{*}$.}
Solving the Riemann problem for system (\ref{bfe:cons}) requires solving a nonlinear algebraic equation  

\begin{equation}  \label{bfe:equation}
          f(A)=0  \;,
\end{equation} 
with root $A_*$ and function $f$
\begin{equation}
    f(A;A_L,A_R) = f_L(A;A_L) + f_R(A; A_R) + u_R - u_L\,, \label{bfe:starprob}
\end{equation}
where $(A_L,u_L)$ and $(A_R,u_R)$ are the left/right states for the Riemann problem and
\begin{equation} \label{bfe:starprobSubFunc}
    f_K(A;A_K) = \left\{ \begin{array}{lllll}
         f^{RAR}_K &= 4 (c(A)-c(A_K))\;, & \mbox{if } & A<A_K   & (rarefaction) \,,\\
         f^{SHO}_K &= \left( \gamma \frac{(A-A_K)(A^\frac{3}{2}-A^\frac{3}{2}_K)}{A A_K} \right)^{\frac{1}{2}}\;, & \mbox{if } & A\geq A_K  & (shock)\,. 
    \end{array}\right.
\end{equation}

\item {\bf Two-rarefaction solution for area $A_{*rr}$.}
When both branches in (\ref{bfe:starprob})-(\ref{bfe:starprobSubFunc}) are those of rarefaction waves then we have
\begin{equation} \label{bfe:twoRarFunc}
    f^{RAR}(A;A_L,A_R) = f^{RAR}_L(A;A_L) + f^{RAR}_R(A; A_R) + u_R - u_L\,
\end{equation}
and the closed-form solution of $f^{RAR}(A;A_L,A_R)=0$  is
\begin{equation}\label{bfe:atr}
 A_{*rr} = \left\{ \frac{2\rho\left[\frac{1}{2}(c_L+c_R)-\frac{1}{8}(u_R-u_L)\right]^2}{\beta}\right\}^2\;.
\end{equation}
Note that the general case (\ref{bfe:atr}) is an approximation to exact solution of  (\ref{bfe:equation}).
\end{enumerate}
\end{lemma}
\begin{proof}
       Omitted. Full details are found in \cite{Toro:2016a}.
\end{proof}

\begin{lemma} \label{bfe:theorar}
 $f(A;A_L,A_R) \geq f^{RAR}(A;A_L,A_R)\;, \quad \forall A>0\,.$
\end{lemma}
\begin{proof}
The approach followed here is similar to the one proposed in \cite{Guermond:2015a}.
 First note that it is sufficient to prove that $f^{SHO}_K \geq f^{RAR}_K$. In fact, fixing $f_K=f^{RAR}_K$ with $f^{SHO}_K \geq f^{RAR}_K$ for either of the two $f_K$ in (\ref{bfe:starprob}) fulfills the statement to be proved.
 We thus concentrate in proving that
 \begin{equation}
     \left( \gamma \frac{(A-A_K)(A^\frac{3}{2}-A^\frac{3}{2}_K)}{A A_K} \right)^{\frac{1}{2}} \geq 4 (c(A)-c(A_K)) \,. \label{eq:toprove}
 \end{equation}
 Before proceeding we introduce $y=A/A_K$ and note that 
 \begin{equation}
     4 (c(A)-c(A_K)) = 4 (\zeta A^{\frac{1}{4}}-\zeta A^{\frac{1}{4}}_K) = 4 \zeta A^{\frac{1}{4}}_K (y^{\frac{1}{4}} -1)\,. \label{eq:rarSimp}
 \end{equation}
 Moreover, we note that 
 \begin{equation}
      \left( \gamma \frac{(A-A_K)(A^\frac{3}{2}-A^\frac{3}{2}_K)}{A A_K} \right)^{\frac{1}{2}} = \left( \gamma  \frac{A_K (y-1) A^\frac{3}{2}_K (y^\frac{3}{2}-1)}{A A_K} \right)^{\frac{1}{2}}  =  \left( \gamma  \frac{ (y-1)  (y^\frac{3}{2}-1)}{y} \right)^{\frac{1}{2}}  A^\frac{1}{4}_K\,.
       \label{eq:shockSimp}
 \end{equation}
 By replacing (\ref{eq:rarSimp}) and (\ref{eq:shockSimp}) into (\ref{eq:toprove}) we obtain
 \begin{equation}
      \left( \gamma \frac{ (y-1)  (y^\frac{3}{2}-1)}{y} \right)^{\frac{1}{2}}  A^\frac{1}{4}_K \geq 4 \zeta A^{\frac{1}{4}}_K (y^{\frac{1}{4}} -1)\,,
 \end{equation}{}
 which proves the claim for $y\leq 1$. Next, for $y>1$, we proceed as follows
 \begin{equation}
       \frac{ (y-1)  (y^\frac{3}{2}-1)}{y}    \geq 16 \frac{\zeta^2}{\gamma}  (y^{\frac{1}{4}} -1)^2\,,
 \end{equation}{}
 noting that 
 \begin{equation}
     \frac{\zeta^2}{\gamma} = \frac{\beta}{2\rho} \frac{3\rho}{\beta} = \frac{3}{2}\,
 \end{equation}{}
 we have that 
 \begin{equation}
       \frac{1}{y} (y-1)  (y^\frac{3}{2}-1)    \geq 24  (y^{\frac{1}{4}} -1)^2\, 
 \end{equation}
 and rearranging
 \begin{equation}
     (y-1) (y^\frac{3}{2}-1) \geq 24 y (y^\frac{1}{4}-1)^2\,,\label{eq:toprove2}
 \end{equation}{}
  We now rewrite the above expression as
 \begin{equation}
     (y-1) (y^\frac{3}{2}-1) - 24 y (y^\frac{1}{4}-1)^2 \geq 0\,
 \end{equation}{}
 and factor it as
\begin{equation}
     (y^\frac{1}{4}-1)^4 (1+3 y^\frac{1}{4}+y^\frac{1}{2})(1+y^\frac{1}{4}+6 y^\frac{1}{2}+y^\frac{3}{4}+y) \geq 0\,,
 \end{equation}{}
 which holds for $y>0$ since all terms are larger or equal to zero for $y>0$. This concludes the proof.
 \end{proof}

\begin{lemma} \label{bfe:theoshockspeed}
 Function $q_K$ in (\ref{bfe:13}) and (\ref{bfe:15}), with $K=\{L,R\}$, is monotone increasing in $A\;,\forall A>0\,.$
\end{lemma}
\begin{proof}
 Differentiation of $q_K$ with respect to $y$ and further manipulation results in 
 \begin{equation}
  q'_K = \frac{3 y^{3/2}+6 y+4 \sqrt{y}+2}{2 \sqrt{6} \left(\sqrt{y}+1\right)^2 \sqrt{\frac{y \left(y+\sqrt{y}+1\right)}{\sqrt{y}+1}}}\;,
 \end{equation}
 which satisfies
 \begin{equation}
  q'_K \geq 0\;, \forall y>0\;
 \end{equation}
and the results follows.
 \end{proof}

\begin{lemma} \label{bfe:theoConcaveDown}
  The function (\ref{bfe:starprob})-(\ref{bfe:starprobSubFunc}) is concave down $\forall A>0\,.$
\end{lemma}
\begin{proof}
  We first prove that both branch functions in (\ref{bfe:starprobSubFunc}) are concave down. In particular, for the rarefaction branch we have that 
 \begin{equation}
  (f^{RAR}_K)'' = -\frac{3}{4} \frac{\zeta}{ A^{7/4}} < 0\;,\quad \forall A>0\;.
 \end{equation}
Now we focus on the shock branch of the function, noting that for this case $y=\frac{A}{A_K}>1$. First, we find that the second derivative of $f^{SHO}_K$ is
\begin{equation} \label{bfe:shoConcavfsec}
  (f^{SHO}_K)'' = \frac{\gamma}{2}\left(  \frac{\gamma(y-1)(y^{\frac{3}{2}}-1)}{y}  \right)^{-\frac{1}{2}}  b   \;,
\end{equation}
with
\begin{equation}\label{bfe:shoConcavb}
 b =  -\frac{y}{2(y-1)(y^{\frac{3}{2}}-1)} \left( \frac{ \frac{3}{2}y^{\frac{5}{2}}-\frac{1}{2}y^{\frac{3}{2}} -1 }{y^2}\right)^2+  \frac{\frac{3}{4}y^{\frac{5}{2}}+ \frac{1}{4}y^{\frac{3}{2}}+2}{y^3}\;.
\end{equation}
Noting that in (\ref{bfe:shoConcavfsec})
\begin{equation}
 \frac{\gamma}{2}\left(  \frac{\gamma(y-1)(y^{\frac{3}{2}}-1)}{y}  \right)^{-\frac{1}{2}} > 0\;,\quad \forall y>1\;,
\end{equation}
we focus on the sign of $b$. Moreover, we observe that all denominators in $b$ are positive for $y>1$. Then we consider the sign of its numerator, which can be written as
\begin{equation}
 b_{num} = (y^\frac{1}{2}-1)^3(-3y^\frac{7}{2}-9y^3-16y^\frac{5}{2}-30y^2-54y^\frac{3}{2}-56y-36y^\frac{1}{2}-12)\;.
\end{equation}
This is clearly negative for $y>1$. Therefore the claim that function (\ref{bfe:starprob})-(\ref{bfe:starprobSubFunc}) is concave down has been proved.

 \end{proof}
 
 In the remaning part fo this section we propose four new estimates for wave speed bounds, which we call: $TMS_a$, $TMS_b$, $TMS_c$ and $TMS_d$. We now proceed to formulate each one of these estimates and to prove our statements. Before proceeding with the definition of wave speed estimates, we first define some quantities that will be used repeatedly, namely
\begin{equation}
 A_{min}=\min\{A_L,A_R\}\;, \hspace{4mm} A_{max} = \max\{A_L,A_R\}\;
\end{equation}
and
\begin{equation}
 f_{min}=f(A_{min})\;, \hspace{4mm} f_{max} = f(A_{max})\;.
\end{equation}
We also define the following three points:
\begin{equation}                                                 \label{eeTMSc4} 
    P_{m}=(A_{min}, f(A_{min})) \;, \hspace{3mm} P_{M}=(A_{max}, f(A_{max})) \;,  \hspace{3mm} P_{rr}=(A_{*rr}, f(A_{*rr})) \;.
\end{equation}
As for the Euler equations and the shallow water equations, these points will be selectively used for linear interpolation to obtain and approximation of the area $A_{*}$.

\subsection{BFE speed bound estimate: approach $TMS_a$}

This estimate requires up to three evaluations of (\ref{bfe:starprob}). Table \ref{tab:bfeTMSc} summarizes the results presented here.

\begin{theorem} \label{bfe:theoTMSc}
 
\noindent{\bf Method $TMS_a$ for BFE:} The exact left and right 
wave speeds $S_{L}^{Ex}$ and $S_{R}^{Ex}$ are bounded by $S_{L}^{b}$ and $S_{R}^{b}$, computed as follows:
\begin{itemize}
 \item{\bf Case R/R: two rarefaction waves.} In this case we have that  $f_{min}\geq0$. Then we set
 \begin{equation} \label{bfe:tsmcRR}
  S_{L}^{b} = S^{Ex}_L = u_L-c_L\;, \hspace{5mm} S_{R}^{b} = S_{R}^{Ex} = u_R+c_R\;.
 \end{equation}
 \item{\bf Case R/S: left rarefaction and right shock.} The conditions are: $f_{min}<0$, $f_{max}>0$ and $A_{max}=A_L$. Then we compute
     \begin{equation}\label{bfe:tsmcSRLin}
        A_{*mM} = A_{min} - \frac{A_{max}-A_{min}}{f_{max}-f_{min}} f_{min}\;,
    \end{equation}
    and then set
     \begin{equation}\label{bfe:tsmcRS}
        S_{L}^{b} = u_L- c_L\;, \hspace{5mm} S_{R}^{b} = u_R + c_R q_{R}(A_{*mM})\;.
     \end{equation}

 \item{\bf Case S/R: left shock and right rarefaction.} The conditions are:  $f_{min}<0$, $f_{max}>0$ and $A_{max}=A_R$. Then we set
     \begin{equation}\label{bfe:tsmcSR}
       S_{L}^{b} = u_L - c_L q_{L}(A_{*mM}) \;, \hspace{5mm} S_{R}^{b} = u_R + c_R\;.
     \end{equation}

\item{\bf Case S/S: two shock waves.} If $f_{max}\leq0$ we compute
 \begin{equation}\label{bfe:tsmcSSLin}
        A_{*Mrr} = A_{max} -  \frac{A_{*rr}-A_{max}}{f(A_{*rr};A_L,A_R)-f_{max}} f_{max}\;
    \end{equation}
 and set
 \begin{equation}\label{bfe:tsmcSS}
  S_{L}^{b} = u_L - c_L q_{L}(A_{*Mrr})\;, \hspace{5mm} S_{R}^{b} = u_R + c_R q_{R}(A_{*Mrr})\;.
 \end{equation}
\end{itemize}
\end{theorem}
\begin{proof}

We consider the four possible wave configurations separately.

\paragraph{$TMS_a$ - Proof for case of left rarefaction/right rarefaction (R/R)}
 
Wave speed estimates (\ref{bfe:tsmcRR}) are identical to the exact wave speeds for this wave configuration.

\paragraph{$TMS_a$ - Proof for case of left rarefaction/right shock (R/S)}
 
Wave speed estimates in (\ref{bfe:tsmcRS}) for $S^b_L$ is identical to the exact wave speed for this wave type.
Wave speed estimate in (\ref{bfe:tsmcRS}) for $S^b_R$ is computed with the same expression used for the exact wave speed for this wave type, see (\ref{bfe:15}), but using $A_{*mM}$ instead of $A_*$. The resulting estimate is a bound since for this wave configuration we have that (\ref{bfe:15}) is monotone increasing in $A$ (from Lemma \ref{bfe:theoshockspeed}) and $f(A;A_L,A_R)$ is concave down (from Lemma \ref{bfe:theoConcaveDown}).
 
\paragraph{$TMS_a$ - Proof for case of left shock/right rarefaction (S/R)}
 
Proving that expressions in (\ref{bfe:tsmcSR}) are bounds is entirely analogous to the previous case.
 
\paragraph{$TMS_a$ - Proof for case of left shock/right shock (S/S)}

Wave speed estimates in (\ref{bfe:tsmcRS}) are computed with the same expression used for the exact wave speeds for this wave type, see (\ref{bfe:15}) , but using $A_{*Mrr}$ instead of $A_*$. The resulting estimates are bounds since for this wave configuration we have that:  $A_{*rr}>A_*$ (from Lemma \ref{bfe:theorar}), (\ref{bfe:15}) is monotone increasing in $A$ (from Lemma \ref{bfe:theoshockspeed}) and $f(A;A_L,A_R)$ is concave down (from Lemma \ref{bfe:theoConcaveDown}).
 
\end{proof}

\vspace{2mm}
\begin{table}
 
\begin{center}
\begin{tabular}{|c|c|c|c|} \hline
Wave pattern   &   Conditions       &    $S_{L}^{b}$     &   $S_{R}^{b}$  \\ \hline
$R/R$   &  $f(A_{min})\ge 0$     &    $u_{L} - c_{L}$                    &   $u_{R} + c_{R}  $  \\ \hline 
$R/S$   &  $f(A_{min}) < 0$ \;, $f(A_{max}) > 0$ \;, $A_{min}=A_{R}$ &  $u_{L} -  c_{L}$  &  $ u_{R} +  c_{R}q_{R}(A_{*mM}) $ \\ \hline
$S/R$ &  $f(A_{min}) < 0$ \;, $f(A_{max}) > 0$ \;, $A_{min}=A_{L}$    &  $ u_{L} -  c_{L}q_{L}(A_{*mM}) $  & $  u_{R} +  c_{R} $ \\ \hline
$S/S$  &  $f(A_{max}) < 0$  &   $  u_{L} -  c_{L}q_{L}(A_{*Mrr})   $           &$   u_{R} + c_{R}q_{R}(A_{*Mrr})  $  \\ \hline 
\end{tabular}       
\caption{$TMS_a$ bound estimates $S_{L}^{b}$ and $S_{R}^{b}$ on minimal and maximal wave speeds for the blood flow equations (\ref{bfe:cons}). Function $q_{K}(A)$ ($K=L, R$) given in Eqs. (\ref{bfe:13})-(\ref{bfe:15}). Value $A_{*mM}$ is given in Eq. (\ref{bfe:tsmcSRLin}) and value $A_{*Mrr}$ is given in Eq. (\ref{bfe:tsmcSSLin}).}\label{tab:bfeTMSc}
\end{center}

\end{table}
\vspace{0.2cm}

\subsection{BFE speed bound estimate: approach $TMS_b$}

This estimate is similar to $TMS_a$. It requires up to three evaluations of function (\ref{bfe:starprob}). Table \ref{tab:bfeTMSd} summarizes the results presented here.

\begin{theorem} \label{bfe:theoTMSd}
 
\noindent{\bf Method $TMS_b$ for BFE:} The exact left and right 
wave speeds $S_{L}^{Ex}$ and $S_{R}^{Ex}$ are bounded by $S_{L}^{b}$ and $S_{R}^{b}$, computed as follows:
\begin{itemize}
    \item{\bf Case R/R: two rarefaction waves.} In this case we have that  $f_{min}\geq0$. Then we set
    \begin{equation} \label{bfe:tsmdRR}
    S_{L}^{b} = S^{Ex}_L = u_L-c_L\;, \hspace{5mm} S_{R}^{b} = S_{R}^{Ex} = u_R+c_R\;.
    \end{equation}
    \item{\bf Case R/S: left rarefaction and right shock.} The conditions are: $f_{min}<0$, $f_{max}>0$ and $A_{max}=A_L$. Then we compute
    \begin{equation}\label{bfe:tsmdLin}
        A_{*mrr} = A_{min} -  \frac{A_{*rr}-A_{min}}{f(A_{*rr};A_L,A_R)-f_{min}} f_{min}\;.
    \end{equation}
    Then we set
    \begin{equation}\label{bfe:tsmdRS}
        S_{L}^{b} = u_L- c_L\;, \hspace{5mm} S_{R}^{b} = u_R + c_R q_{R}(A_{*mrr})\;.
    \end{equation}
    \item{\bf Case S/R: left shock and right rarefaction.} The conditions are:  $f_{min}<0$, $f_{max}>0$ and $A_{max}=A_R$. Then we set
    
    \begin{equation}\label{bfe:tsmdSR}
    S_{L}^{b} = u_L - c_L q_{L}(A_{*mrr}) \;, \hspace{5mm} S_{R}^{b} = u_R + c_R\;.
    \end{equation}
        
\item{\bf Case S/S: two shock waves.} If $f_{max}\leq0$ we set
 \begin{equation}\label{bfe:tsmcSS}
  S_{L}^{b} = u_L - c_L q_{L}(A_{*mrr})\;, \hspace{5mm} S_{R}^{b} = u_R + c_R q_{R}(A_{*mrr})\;.
 \end{equation}
\end{itemize}
\end{theorem}
\begin{proof}

We consider the four possible wave configurations separately.

\paragraph{$TMS_b$ - Proof for case of left rarefaction/right rarefaction (R/R)}
 
Wave speed estimates (\ref{bfe:tsmdRR}) are identical to the exact wave speeds for this wave configuration.

\paragraph{$TMS_b$ - Proof for case of left rarefaction/right shock (R/S)}
 
Wave speed estimates in (\ref{bfe:tsmdRS}) for $S^b_L$ is identical to the exact wave speed for this wave type.
Wave speed estimate in (\ref{bfe:tsmdRS}) for $S^b_R$ is computed with the same expression used for the exact wave speed for this wave type, see (\ref{bfe:15}) , but using $A_{*mrr}$ instead of $A_*$. The resulting estimate is a bound since for this wave configuration we have that: $A_{*rr}>A_*$ (from Lemma \ref{bfe:theorar}), (\ref{bfe:15}) is monotone increasing in $A$ (from Lemma \ref{bfe:theoshockspeed}) and $f(A;A_L,A_R)$ is concave down (from Lemma \ref{bfe:theoConcaveDown}).
 
\paragraph{$TMS_b$ - Proof for case of left shock/right rarefaction (S/R)}
 
Proving that expressions in (\ref{bfe:tsmdSR}) are bounds is entirely analogous to the previous case.
 
\paragraph{$TMS_b$ - Proof for case of left shock/right shock (S/S)}

Wave speed estimates in (\ref{bfe:tsmdRS}) are computed with the same expression used for the exact wave speeds for this wave type, see (\ref{bfe:15}) , but using $A_{*mrr}$ instead of $A_*$. The resulting estimates are bounds since for this wave configuration we have that:  $A_{*rr}>A_*$ (from Lemma \ref{bfe:theorar}), (\ref{bfe:15}) is monotone increasing in $A$ (from Lemma \ref{bfe:theoshockspeed}) and $f(A;A_L,A_R)$ is concave down (from Lemma \ref{bfe:theoConcaveDown}).
\end{proof}
\begin{table}
\begin{center}
\begin{tabular}{|c|c|c|c|} \hline
Wave pattern   &   Conditions       &    $S_{L}^{b}$     &   $S_{R}^{b}$  \\ \hline
$R/R$   &  $f(A_{min})\ge 0$     &    $u_{L} - c_{L}$                    &   $u_{R} + c_{R}  $  \\ \hline 
$R/S$   &  $f(A_{min}) < 0$ \;, $f(A_{max}) > 0$ \;, $A_{min}=A_{R}$ &  $u_{L} -  c_{L}$  &  $ u_{R} +  c_{R}q_{R}(A_{*mrr}) $ \\ \hline
$S/R$ &  $f(A_{min}) < 0$ \;, $f(A_{max}) > 0$ \;, $A_{min}=A_{L}$    &  $ u_{L} -  c_{L}q_{L}(A_{*mrr}) $  & $  u_{R} +  c_{R} $ \\ \hline
$S/S$  &  $f(A_{max}) < 0$  &   $  u_{L} -  c_{L}q_{L}(A_{*mrr})   $           &$   u_{R} + c_{R}q_{R}(A_{*mrr})  $  \\ \hline 
\end{tabular}       
\caption{$TMS_b$ bound estimates $S_{L}^{b}$ and $S_{R}^{b}$ on minimal and maximal wave speeds for the blood flow equations (\ref{bfe:cons}). Function $q_{K}(A)$ ($K=L, R$) given in Eqs. (\ref{bfe:13})-(\ref{bfe:15}). Value  $A_{*mrr}$ is given in Eq. (\ref{bfe:tsmdLin}).}\label{tab:bfeTMSd}
\end{center}
\end{table}

\subsection{BFE wave bound estimate: approach $TMS_c$}

This estimate requires two evaluations of function (\ref{bfe:starprob}) in order to determine the wave pattern configuration related to a given Riemann problem. Table \ref{tab:bfeTMSb} summarizes the results presented here. 

\begin{theorem} \label{bfe:theoTMSb}
 
\noindent{\bf Method $TMS_c$ for BFE:} The exact left and right 
wave speeds $S_{L}^{Ex}$ and $S_{R}^{Ex}$ are bounded by $S_{L}^{b}$ and $S_{R}^{b}$, computed as follows:
\begin{itemize}
 \item{\bf Case R/R: two rarefaction waves.} In this case we have that  $f_{min}\geq0$. Then we set 
 \begin{equation} \label{bfe:tsmbRR}
  S_{L}^{b} = S^{Ex}_L = u_L-c_L\;, \hspace{5mm} S_{R}^{b} = S_{R}^{Ex} = u_R+c_R\;.
 \end{equation}
 \item{\bf Case R/S: left rarefaction and right shock.} The conditions are: $f_{min}<0$, $f_{max}>0$ and $A_{max}=A_L$. Then we consider
     \begin{equation}\label{bfe:tsmbRS}
        S_{L}^{b} = u_L- c_L\;, \hspace{5mm} S_{R}^{b} = u_R + c_R q_{R}(A_L)\;.
     \end{equation}
 \item{\bf Case S/R: left shock and right rarefaction.} The conditions are: $f_{min}<0$, $f_{max}>0$ and $A_{max}=A_R$. Then one sets   
     
     \begin{equation}\label{bfe:tsmbSR}
       S_{L}^{b} = u_L - c_L q_{L}(A_R) \;, \hspace{5mm} S_{R}^{b} = u_R + c_R\;.
     \end{equation}
\item{\bf Case S/S: two shock waves.} If $f_{max}\leq0$ one sets
 \begin{equation}\label{bfe:tsmbSS}
  S_{L}^{b} = u_R- c_R\;, \hspace{5mm} S_{R}^{b} = u_L + c_L\;.
 \end{equation}
 
\end{itemize}
\end{theorem}
\begin{proof}

We consider the four possible wave configurations separately.

\paragraph{$TMS_c$ - Proof for case of left rarefaction/right rarefaction (R/R)}
 
Wave speed estimates (\ref{bfe:tsmbRR}) are identical to the exact wave speeds for this wave configuration.

\paragraph{$TMS_c$ - Proof for case of left rarefaction/right shock (R/S)}
 
Wave speed estimates in (\ref{bfe:tsmbRS}) for $S^b_L$ is identical to the exact wave speed for this wave type.
Wave speed estimate in (\ref{bfe:tsmbRS}) for $S^b_R$ is computed with the same expression used for the exact wave speed for this wave type, see (\ref{bfe:15}) , but using $A_L$ instead of $A_*$. The resulting estimate is a bound since for this wave configuration $A_L>A_*$ and (\ref{bfe:15}) is monotone increasing in $A$ (see Lemma \ref{bfe:theoshockspeed}).
 
\paragraph{$TMS_c$ - Proof for case of left shock/right rarefaction (S/R)}
 
Proving that expressions in (\ref{bfe:tsmbSR}) are bounds is entirely analogous to the previous case. 
 
\paragraph{$TMS_c$ - Proof for case of left shock/right shock (S/S)}
 
The Lax entropy condition ensures that 
\begin{equation}                                                 \label{bfe:28} 
		u_{*} + c_{*} > S_{R}^{Ex}\;.
\end{equation}
The task is to find a bound for the characteristic speed $u_{*} + c_{*} $ in the start region.  Recalling our hypothesised bound, we want to prove that
\begin{equation}                                                 \label{bfe:29} 
		S_{R}^{b} = u_{L} + c_{L}  \ge u_{*} + c_{*}   \;.
\end{equation}
Here we recall  (\ref{bfe:12}) for a left shock wave
\begin{equation}                                                 \label{bfe:30}
		u_{*} = u_{L} -  \sqrt{\frac{\gamma (A_{*}-A_{L})(A_{*}^{3/2}-A_{L}^{3/2})}{A_{L}A_{*}}} = u_L - \left( \gamma  \frac{ (y-1)  (y^\frac{3}{2}-1)}{y} \right)^{\frac{1}{2}}  A^\frac{1}{4}_L\;,
		\hspace{3mm}  y \equiv \frac{A^{*}}{A_{L}}\;.
\end{equation}
Substitution of  $u_{*}$  into (\ref{bfe:29}) gives
\begin{equation}                                                 \label{bfe:31a} 
		c_L \ge -  \left( \gamma  \frac{ (y-1)  (y^\frac{3}{2}-1)}{y} \right)^{\frac{1}{2}}  A^\frac{1}{4}_L + c\;.
\end{equation}
Recalling that $c=\zeta A^\frac{1}{4}$ and $\gamma/\zeta^2=2/3$, after some algebraic manipulations we obtain
\begin{equation}                                                 \label{bfe:31} 
		y^{\frac{1}{4}} - \left[\frac{2}{3}\frac{(y-1)(y^\frac{3}{2} - 1 )}{y}\right]^{\frac{1}{2}}   \le 1 \;,
		\hspace{3mm}  y \equiv \frac{A^{*}}{A_{L}} \;.
\end{equation}
which can be written as
\begin{equation}
 y^\frac{1}{4} -1 \leq \left[\frac{2}{3}\frac{(y-1)(y^\frac{3}{2} - 1 )}{y}\right]^{\frac{1}{2}}\;.
\end{equation}
Since we are studying a S/S condition we have that $y>1$, which means that the left-hand side of the above inequality is always positive. We also note that the inequality holds for every $y \leq 1$. We then further manipulate the expression, obtaining
\begin{equation}
 (y-1)(y^\frac{3}{2}-1)-\frac{3}{2}y(y^\frac{1}{4}-1)^2 \geq 0\;,
\end{equation}
which can be factored to
\begin{equation}
 \frac{1}{2}(y^\frac{1}{4}-1)^2(2+4y^\frac{1}{4}+6y^\frac{1}{2}+8y^\frac{3}{4}+5y+8y^\frac{5}{4}+6y^\frac{3}{2}+4y^\frac{7}{4}+2y^2) \geq 0\,,
\end{equation}
which clearly holds for $y\geq0$. Hence, we have that expression $ S^b_R$ in (\ref{bfe:tsmbSS}) clearly results in a bound. The proof for $S^b_L$ is analogous to the one performed for $S^b_R$ and is thus omitted. With this last case we have covered all four possible configurations and the proof is complete.
 
\end{proof}

\vspace{2mm}
\begin{table}
 
\begin{center}
\begin{tabular}{|c|c|c|c|} \hline
Wave pattern   &   Conditions       &    $S_{L}^{b}$     &   $S_{R}^{b}$  \\ \hline
$R/R$   &  $f(A_{min})\ge 0$     &    $u_{L} - c_{L}$                    &   $u_{R} + c_{R}  $  \\ \hline 
$R/S$   &  $f(A_{min}) < 0$ \;, $f(A_{max}) > 0$ \;, $A_{min}=A_{R}$ &  $u_{L} -  c_{L}$  &  $ u_{R} +  c_{R}q_{R}(A_{L}) $ \\ \hline
$S/R$ &  $f(A_{min}) < 0$ \;, $f(A_{max}) > 0$ \;, $A_{min}=A_{L}$    &  $ u_{L} -  c_{L}q_{L}(A_{R}) $  & $  u_{R} +  c_{R} $ \\ \hline
$S/S$  &  $f(A_{max}) < 0$  &   $  u_{R} -  c_{R}   $           &$   u_{L} + c_{L}$  \\ \hline 
\end{tabular}       
\caption{$TMS_c$ bound estimates $S_{L}^{b}$ and $S_{R}^{b}$ on minimal and maximal wave speeds for the blood flow equations (\ref{bfe:cons}). Function $q_{K}(A)$ ($K=L, R$) given in Eqs. (\ref{bfe:13})-(\ref{bfe:15}). }\label{tab:bfeTMSb}
\end{center}

\end{table}
\vspace{0.2cm}

\subsection{BFE speed bound estimate: approach $TMS_d$}

This estimate is a function of left and right state vectors and uses very little information about the Riemann problem.

\begin{theorem}  \label{bfe:theoTMSa}
\noindent{\bf Method $TMS_d$ for BFE:} The exact left and right 
wave speeds $S_{L}^{Ex}$ and $S_{R}^{Ex}$ are bounded by $S_{L}^{b}$ and $S_{R}^{b}$  respectively, with
\begin{equation}                           \label{bfe:16} 
		S_{L}^{b} = \min\{u_{L} - c_{L},u_{R} - \alpha_L c_{R}\} \le S_{L}^{e} \;, \hspace{3mm} \alpha_{L}=4
\end{equation}
and
\begin{equation}                            \label{bfe:17} 
		S_{R}^{b} = \max\{u_{R} + c_{R}, u_{L} + \alpha_{R} c_{L} \} \ge S_{R}^{b}\;, \hspace{3mm} \alpha_{R}=4 \;.
\end{equation}
\end{theorem}

\begin{proof}

We consider the four possible wave configurations separately.

\paragraph{$TMS_d$ - Proof for case of left rarefaction/right rarefaction (R/R)}

The proof in this case is trivial. If the exact solution consists of two rarefaction waves, then the exact wave speeds will be
\begin{equation}
 S_{L}^{Ex} = u_L-c_L\;, \hspace{5mm} S_{R}^{Ex} = u_R+c_R\;.
\end{equation}

Clearly (\ref{bfe:16}) and (\ref{bfe:17}) are bounds to these wave speeds.

\paragraph{$TMS_d$ - Proof for case of left rarefaction/right shock (R/S)} 

Application of the Lax entropy condition gives
\begin{equation}                             \label{bfe:18} 
		u_{*} + c_{*} > S_{R}^{Ex}\;.
\end{equation}
The task is to find a bound for the characteristic speed $u_{*} + c_{*} $ in the start region.  We seek a bound of the form
\begin{equation}                              \label{bfe:19} 
		S_{R}^{b} = u_{L} + \alpha_{R} c_{L}  \ge u_{*} + c_{*}   \;,
\end{equation}
with $\alpha_{R}=4$ yet to be justified. From the left Riemann invariants (\ref{bfe:10}) we may write
\begin{equation}                               \label{bfe:20} 
		u_{*} = u_{L} - 4 (c_{*}-c_{L}) \;
\end{equation}
and thus the characteristic speed becomes
\begin{equation}                                \label{bfe:21} 
	 u_{*} +c_{*} = u_{L} + 4 c_{L} - 3c_{*} \;.
\end{equation}
By imposing (\ref{bfe:19}) we have
\begin{equation}                                \label{bfe:22} 
		u_{L} +4 c_{L} - 3c_{*}  \le u_{L} + \alpha_{R} c_{L} \;.
\end{equation}
Simple manipulations lead to
\begin{equation}                                \label{bfe:23} 
		4 - 3 y^\frac{1}{4} \le \alpha_{R}\;, \hspace{3mm}  y=\frac{A_{*}}{A_{L}} \;.
\end{equation}
Since the left wave is a rarefaction, we have that $0 < y \le 1$. The most restrictive case is $\alpha_R=4$ and the result follows.\\

\paragraph{$TMS_d$ - Proof for case of left shock / right rarefaction (S/R)} 
The proof for this configuration is analogous to the previous one for the R/S configuration and it is thus omitted.
\paragraph{$TMS_d$ - Proof for case of left shock / right shock (S/S)} 
The proof is analogous to the one proposed for this wave configuration in Theorem (\ref{bfe:theoTMSb}) and it is thus omitted here. 
\end{proof}

\subsection{Numerical tests for the blood flow equations}

Six Riemann problems have been selected to test the performance of the existing and newly proposed wave speed estimates for the blood flow equations. Table  \ref{Tab:BFE_Tests_IC} shows the initial condition in terms of primitive variables. For all tests problems we have considered a blood vessel with the following properties: $\rho=1.05\;g/cm^3$; $\beta=28209.4792\;dyne/cm^3$. Exact solution for cross-sectional area $A_{*}$ and velocity $u_{*}$ are shown in columns $6$ and $7$ respectively, while the emerging wave pattern  for each test is shown in the last column.

Tables \ref{Tab:BFE_Tests_SR} and \ref{Tab:BFE_Tests_SL} show results for the maximal and minimal wave speeds respectively. The numerical results confirm that all proposed estimates constitute bounds for the extreme wave speeds, for the chosen tests. Moreover, we observe that estimates (\ref{Davis:a}) and (\ref{Davis:b}) fail for some tests, with larger errors for shock waves. Bounds that rely heavily on the two-rarefaction solution perform very well, starting to deviate significantly from exact values when strong shock waves are present, as for example in test 3. This behaviour is expected. Regarding the simpler bounds $TMS_c$ and $TMS_d$  proposed here, we observe that $TMS_c$ gives more accurate estimates than $TMS_d$ for almost all cases, except for Test 3, where the shock speed is grossly overestimated. In general, $TMS_a$ and $TMS_b$ perform better than other estimates. In Test 3 one can see the improvement resulting from using a linear interpolation and not just the two-rarefaction solution when comparing $To-$ and $GP-$ type estimates with $TMS_b$.
\begin{table}[h!]    
\begin{center}
\vspace{1mm}
\begin{tabular}{|c|c|c|c|c|c|c|c|} \hline
Test &  $A_{L}$  & $u_{L}$  & $A_{R}$  & $u_{R}$ & $A_{*}$  & $u_{*}$& wave pattern  \\ \hline
1	&	3.1416	&	0.0000	&	2.8274	&	0.0000	&	2.9814	&	8.0234	&	rar-shock	\\ \hline
2	&	3.1416	&	0.0000	&	0.6283	&	0.0000	&	1.4936	&	104.6900	&	rar-shock	\\ \hline
3	&	0.0031	&	0.0000	&	3.1416	&	0.0000	&	0.1220	&	-343.2400	&	shock-rar	\\ \hline
4	&	3.1416	&	10.0000	&	3.1416	&	-20.0000	&	3.4581	&	-5.0000	&	shock-shock	\\ \hline
5	&	3.1416	&	100.0000	&	2.5133	&	-200.0000	&	6.6064	&	-30.6180	&	shock-shock	\\ \hline
6	&	3.1416	&	-10.0000	&	3.1416	&	20.0000	&	2.8471	&	5.0000	&	rar-rar	\\ \hline
\end{tabular}     
\caption{Initial conditions for all six Riemann problems for the blood flow equations. Exact solution for cross-sectional area $A_{*}$ and velocity $u_{*}$ are shown in columns $6$ and $7$ respectively.  The emerging wave pattern  for each test is shown in the last column. Units are: $[A]= cm^2$, $[u]=cm/s$.}
\label{Tab:BFE_Tests_IC}     
\end{center}
\end{table}

\vspace{0.1cm}

\begin{table}[h!]
\begin{center}
\footnotesize
\vspace{1mm}
\begin{tabular}{|c|c|c|c|c|c|c|c|c|c|} \hline
Test & $S^{Ex}_{R}$  & $S^{Dav_a}_{R}$ & $S^{Dav_b}_{R}$   & $S^{To}_{R}$ & $S^{GP}_R$& ${S}^{TMS_a}_{R}$  & ${S}^{TMS_b}_{R}$ & ${S}^{TMS_c}_{R}$ & ${S}^{TMS_d}_{R}$\\ \hline
	
1	&	155.3674	&	\textcolor{red}{\bf 150.292}	&	\textcolor{red}{\bf 154.3033}	&	155.3678	&	155.3678	&	155.4675	&	155.3674	&	160.5683	&	616.7498	\\ \hline
2	&	180.7129	&	\textcolor{red}{\bf 103.1889}	&	\textcolor{red}{\bf154.3033}	&	183.0477	&	183.0477	&	203.5859	&	181.2323	&	300.5546	&	616.7498	\\ \hline
3	&	154.3033	&	154.3033	&	154.3033	&	154.3033	&	154.3033	&	154.3033	&	154.3033	&	154.3033	&	154.3033	\\ \hline
4	&	143.8834	&	\textcolor{red}{\bf134.3033}	&	164.3033	&	143.8893	&	143.8893	&	143.8836	&	143.8836	&	164.3033	&	626.7498	\\ \hline
5	&	73.3884	&	\textcolor{red}{\bf-54.0689}	&	254.3033	&	80.5980	&	80.5980	&	74.5808	&	74.9150	&	254.3033	&	716.7498	\\ \hline
6	&	174.3033	&	174.3033	&	174.3033	&	174.3033	&	174.3033	&	174.3033	&	174.3033	&	174.3033	&	606.7498	\\ \hline

\end{tabular}
\caption{Results for maximal wave speed $S_R$ for tests $1$ to $6$. The exact solution is shown in column 2. Existing estimates are shown in columns  $3$ to $6$. Newly proposed estimates of the present paper are displayed in columns  $7$ to $10$. A value in red indicates thar it fails to bound  the exact wave speed. Units are: $[S]=cm/s$.}
\label{Tab:BFE_Tests_SR}            
\end{center}
\end{table}

\vspace{0.1cm}

\vspace{1mm}
\begin{table}[h!]
\begin{center}
\footnotesize
\vspace{1mm}
\begin{tabular}{|c|c|c|c|c|c|c|c|c|c|} \hline
Test & $S^{Ex}_{L}$  & $S^{Dav_a}_{L}$ & $S^{Dav_b}_{L}$   & $S^{To}_{L}$ & $S^{GP}_L$& ${S}^{TMS_a}_{L}$  & ${S}^{TMS_b}_{L}$ & ${S}^{TMS_c}_{L}$ & ${S}^{TMS_d}_{L}$\\ \hline
1	&	-154.3033	&	-154.3033	&	-154.3033	&	-154.3033	&	-154.3033	&	-154.3033	&	-154.3033	&	-154.3033	&	-600.7046	\\ \hline
2	&	-154.3033	&	-154.3033	&	-154.3033	&	-154.3033	&	-154.3033	&	-154.3033	&	-154.3033	&	-154.3033	&	-412.2919	\\ \hline
3	&	-352.3120	&	\textcolor{red}{\bf-27.4394}	&	\textcolor{red}{\bf-154.3033}	&	-816.8334	&	-816.8334	&	-784.3538	&	-473.8520	&	-3986.0259	&	-616.7498	\\ \hline
4	&	-153.8834	&	\textcolor{red}{\bf -144.3033}	&	-174.3033	&	-153.8893	&	-153.8893	&	-153.8836	&	-153.8836	&	-174.3033	&	-636.7498	\\ \hline
5	&	-149.0520	&	\textcolor{red}{\bf -54.3033}	&	-345.9311	&	-155.5183	&	-155.5183	&	-150.1215	&	-150.4213	&	-345.9311	&	-783.2608	\\ \hline
6	&	-164.3033	&	-164.3033	&	-164.3033	&	-164.3033	&	-164.3033	&	-164.3033	&	-164.3033	&	-164.3033	&	-596.7498	\\ \hline

\end{tabular}
\caption{Results for minimal wave speed $S_L$ for tests $1$ to $6$. The exact solution is shown in column 2. Existing estimates are shown in columns  $3$ to $6$. Newly proposed estimates of the present paper are displayed in columns  $7$ to $10$. A value in red indicates thar it fails to bound  the exact wave speed. Units are: $[S]=cm/s$.}
\label{Tab:BFE_Tests_SL}            
\end{center}
\end{table}

\section{Conclusions}

In this paper we have provided theoretical bound estimates for the two fastest wave speeds emerging from the solution of the Riemann problem for the Euler equations, the shallow water equations and the blood flow equations for arteries. Several new, non-iterative approaches have been presented. The resulting bounds range from crude but simple estimates to accurate but sophisticated estimates that make limited use of information from the exact solution of the Riemann problem. We have assesed our wave speed estimates against exact solutions and against previously proposed wave speed estimates, through a carefully chosen suite of test problems with exact solution. The results confirm that the derived theoretical bounds actually bound from below and above the minimal and maximal wave speeds respectively. The results also show that popular, previously proposed estimates do not bound the true speeds in general.

The results of the present study are directly relevant to the implementation of large classes of explicit numerical methods intended for systems hyperbolic equations, such as finite volume and  discontinuous Galerkin finite element methods, amongs others. Such relevance is seen immediately when enforcing the Courant 
stability condition, in which a reliable estimate of the maximum wave speed is required, where reliable means that the scheme remains within its regions of monotonicity and of linear stability.
Moreover, some numerical methods utilize local wave speed estimates to construct the {\it numerical flux}, such as in HLL-type fluxes, for example. For such methods it is also desirable to utilise wave speed estimates that bound the minimal and maximal wave speeds. As seen in the preliminary discussion in  \cite{Toro:2001a}, Sec. 10.5.2, under estimating the extreme wave speeds may drive the method outside its region of monotonicity, even if remaining within its region of linear stability. All these issues are not adressed in the present paper, but are the subject of current investigations, the results of which will be communicated in the near future.

\newpage

\section*{In memoriam}

This paper is dedicated to the memory of Dr. Douglas Nelson Woods ($^*$January 11\textsuperscript{th} 1985 - $\dagger$September 11\textsuperscript{th} 2019),  
promising young scientist and post-doctoral research fellow at Los Alamos National Laboratory.   
Our thoughts and wishes go to his wife Jessica, to his parents Susan and Tom, to his sister Rebecca and to his brother Chris,  whom he left behind.

\newpage

\begin{center}
{\bf ACKNOWLEDGEMENTS:} 
\end{center}

EFT  acknowledges the partial support received from the European Union's Horizon 2020 Research and  
Innovation Programme under the project \textit{ExaHyPE}, grant agreement number no. 671698 (call FETHPC-1-2014). 
He also acknowledges the financial support received from the Italian Ministry of Education, University and Research (MIUR) in the  
frame of the Departments of Excellence Initiative 2018--2022 attributed to DICAM of the University of Trento (grant L. 232/2016) 
and in the frame of the PRIN 2017 project \textit{Innovative numerical methods for evolutionary partial differential equations and applications}. \\

LOM acknowledges funding from the Italian Ministry of Education, University 
and Research (MIUR) in the frame of the Departments of Excellence Initiative 2018-2022   
attributed to the Department of Mathematics of the University of Trento (grant L. 232/2016) and in the frame of the PRIN 2017 project \textit{Innovative numerical methods for evolutionary partial differential equations and  applications}. 
Furthermore, LM has also received funding from the University of Trento via the Strategic Initiative \textit{Modeling and Simulation}.\\

All three authors gratefully acknowledge the contribution of PhD students Morena Celant and Beatrice Ghitti (Mathematics Department, University of Trento, Italy) in 
tackling some challenging algebraic tasks in the proof of Lemma 9 for the blood flow equations.

\newpage

\begin{center}
{\bf MANUSCRIPT HISTORY:} 
\end{center}

This manuscript was submitted for publication to Computers and Fluids on 3rd February 2020.

\newpage

\bibliography{refs-latest}

\bibliographystyle{plain}

\end{document}